\documentclass[11pt,a4paper]{article}
\usepackage[german,english]{babel}
\usepackage{latexsym}
\usepackage{amsmath, amsfonts}
\usepackage{amssymb}	
\usepackage{amsthm}	
\usepackage{ifthen}
\usepackage{hyperref} 
\usepackage{graphicx} 
\usepackage[arrow,matrix]{xy} 
\usepackage{verbatim} 

\author{Michael Klotz\footnote{Technische Universit\"at Darmstadt, Schlo\ss gartenstra\ss e 7, D-64289 Darmstadt, Deutschland,\newline klotz@mathematik.tu-darmstadt.de}}
\date{}
\title{An Integrability Criterion for Banach--Lie Triple Systems}

\theoremstyle{definition} 
\newtheorem{definition}{Definition}[section]
\theoremstyle{plain} 
\newtheorem{theorem}[definition]{Theorem}	
\newtheorem{proposition}[definition]{Proposition}	
\newtheorem{lemma}[definition]{Lemma}
\newtheorem{corollary}[definition]{Corollary}	
\theoremstyle{definition}
\newtheorem{example}[definition]{Example}
\newtheorem{remark}[definition]{Remark}

\newenvironment{acknowledgements}{\section*{Acknowledgements}}{}

\newcommand{\temp}{0} 
	{\renewcommand{\temp}{\arraystretch} \renewcommand{\arraystretch}{#1}}
	{\renewcommand{\arraystretch}{\temp}}

\newcommand{\NN}{\mathbb N}
\newcommand{\ZZ}{\mathbb Z}

\newcommand{\RR}{\mathbb R}
\newcommand{\CC}{\mathbb C}

\newcommand{\eins}{{\bf 1}}

\newcommand{\calU}{{\cal U}}

\DeclareMathOperator{\ad}{ad}

\DeclareMathOperator{\Aut}{Aut}
\DeclareMathOperator{\const}{const}

\DeclareMathOperator{\Der}{Der}

\DeclareMathOperator{\Exp}{Exp}
\DeclareMathOperator{\ev}{ev}

\DeclareMathOperator{\flow}{Fl}
\DeclareMathOperator{\Fr}{Fr}
\DeclareMathOperator{\GL}{GL}

\DeclareMathOperator{\id}{id}
\DeclareMathOperator{\im}{im}
\DeclareMathOperator{\Inn}{Inn}

\DeclareMathOperator{\Lts}{Lts}
\DeclareMathOperator{\per}{per}

\DeclareMathOperator{\spann}{span}

\DeclareMathOperator{\Sym}{Sym}
\DeclareMathOperator{\Tor}{Tor}

\newcommand{\gl}{\mathfrak{gl}}
\newcommand{\calLts}{\mathfrak{Lts}}

\begin{document}
\maketitle
%
%
%
%
\begin{abstract}
	To give a criterion for the integrability of Banach--Lie triple systems, we follow the construction of the period group of a Lie algebra and define the period group of a Lie triple system as an analogous concept. We show that a Lie triple system is integrable if and only if its period group is discrete. Along the way, we see how to turn the path and the loop space of a pointed symmetric space into pointed symmetric spaces.\\[-0.5\baselineskip]
			
	
\end{abstract}
%
%
%
\section{Introduction}
\label{introduction}
A \emph{symmetric space} in the sense of O.~Loos (cf.\ \cite{Loo69}) is a smooth manifold $M$ endowed with a smooth multiplication map $\mu\colon M\times M\rightarrow M$ such that each left multiplication map $\mu_x:=\mu(x,\cdot)$ (with $x\in M$) is an involutive automorphism of $(M,\mu)$ with isolated fixed\linebreak point $x$.

Some basic material on infinite-dimensional symmetric spaces can be found in \cite{Nee02Cartan} and \cite{Ber08}. In \cite{Klo09b} and \cite{Klo10Subspaces}, the author started working towards a Lie theory of symmetric spaces modelled on Banach spaces. In particular, in \cite{Klo09b}, one finds an integrability theorem for morphisms of Lie triple systems and the result that the automorphism group of a connected symmetric space $M$ is a Banach--Lie group acting smoothly and transitively on $M$. Moreover, in \cite{Klo10Subspaces}, the author deals with the different kinds of reflection subspaces and their Lie triple systems and gives a quotient theorem.
The purpose of this paper is to continue on this path, giving an integrability criterion for Lie triple systems.

In the finite-dimensional case, every Lie triple system is integrable to a pointed symmetric space, i.e., arises as the Lie triple system of a pointed symmetric space. Indeed, given a Lie triple system $\mathfrak{m}$, it can be embedded into a symmetric Lie algebra $S(\mathfrak{m})=S(\mathfrak{m})_+\oplus \mathfrak{m}$ called the standard embedding, which is integrable to a 1-connected symmetric Lie group $(G,\sigma)$, with the consequence that the Lie triple system of the pointed 1-connected symmetric space $M:=G/(G^\sigma)_0$ is isomorphic to $\mathfrak{m}$ (cf.\ \cite[p.~116]{Loo69}).

Since infinite-dimensional Lie algebras are not always integrable, we also expect obstructions concerning the integrability of infinite-dimensional Lie triple systems. In the Banach context, it is a well-known result that a Banach--Lie algebra $\mathfrak{g}$ is integrable if and only if its period group $\Pi(\mathfrak{g})$ (which is a subgroup of the center $\mathfrak{z}(\mathfrak{g})$ of $\mathfrak{g})$ is discrete (cf.\ \cite{EK64} and \cite{GN03}). In this work, we give a similar integrability criterion for Lie triple systems. To this end, we prove that a Lie triple system $\mathfrak{m}$ is integrable if and only if its standard embedding is integrable to a (symmetric) Lie group. Then we define a period group $\Pi(\mathfrak{m})$ (which is a subgroup of the center $\mathfrak{z}(\mathfrak{m})$ of $\mathfrak{m}$) by following closely the corresponding construction of a period group given in \cite{GN03}. We show that $\mathfrak{m}$ is integrable if and only if its period group $\Pi(\mathfrak{m})$ is discrete.

The idea for defining the period group $\Pi(\mathfrak{m})$ is to refine, in a first step, the process leading to the period group $\Pi(\mathfrak{g})$ of a Banach--Lie algebra $\mathfrak{g}$ by considering additional involutive automorphisms and secondly to apply this refinement to the standard embedding of $\mathfrak{m}$. By turning the arising symmetric Lie groups into suitable symmetric pairs, we obtain morphisms of symmetric pairs that induce morphisms of symmetric spaces. More generally, we show that this works not only for the standard embedding, but also for every symmetric Lie algebra $\mathfrak{g}=\mathfrak{g}_+\oplus\mathfrak{g}_-$ with $\mathfrak{g}_-=\mathfrak{m}$ and $\mathfrak{z}(\mathfrak{g})=\mathfrak{z}(\mathfrak{m})$. We observe that $\Pi(\mathfrak{g})\subseteq \Pi(\mathfrak{m})$. One of our main results is then the equivalence of four conditions, namely the integrability of $\mathfrak{m}$, the integrability of $\mathfrak{g}$, the discreteness of $\Pi(\mathfrak{m})$ and the discreteness of $\Pi(\mathfrak{g})$.

Given a pointed 1-connected symmetric space $(M,b)$ with Lie triple system $\mathfrak{m}$, the period group $\Pi(\mathfrak{m})$ can also be computed by a more explicit formula: The exponential map $\Exp_{(M,b)}$ restricts to a morphism $\Exp_{(M,b)}|_{\mathfrak{z}(\mathfrak{m})}$ of pointed symmetric spaces and we obtain $\Pi(\mathfrak{m})=\ker(\Exp_{(M,b)}|_{\mathfrak{z}(\mathfrak{m})})$.

By means of our results, we explain how examples of non-integrable Lie triple systems can be obtained: One source of examples are suitable quotients of integrable Lie triple systems with non-trivial period group. Other examples are obtained from non-integrable Lie algebras. Further, we apply our results to characterize the integrability of complexifications of real Lie algebras: Given a real Banach--Lie algebra $\mathfrak{g}$, its complexification $\mathfrak{g}_\CC:=\mathfrak{g}\oplus i\mathfrak{g}$ is integrable if and only if the Lie algebra $\mathfrak{g}$ and the Lie triple system $i\mathfrak{g}$ both are integrable.

Since the approach of \cite{GN03} is quite geometric, we prepare the required analogous concepts for Banach symmetric spaces:
We provide the basic facts on universal covering spaces and see how to turn the path and the loop space of a pointed symmetric space into pointed symmetric spaces.

To state this more precisely, let $(M,b)$ be a pointed symmetric space. The set $P(M,b):=\{\gamma\in C([0,1],M)\colon \gamma(0)=b\}$ carries a natural structure of a pointed reflection space,
where the multiplication on $P(M,b)$ is defined pointwise and where the base point is given by the constant curve $\const_b$ with value $b$. We observe that it carries a unique structure of a Banach manifold making it a pointed symmetric space with Lie triple system $P(\mathfrak{m}):=\linebreak\{\gamma\in C([0,1],\mathfrak{m})\colon \gamma(0)=0\}$ (where the Lie bracket is defined pointwise) such that $$P(\Exp_{(M,b)})\colon P(\mathfrak{m})\rightarrow P(M,b),\ \gamma \mapsto \Exp_{(M,b)}\circ\gamma$$
is its exponential map. The topology of this \emph{path (symmetric) space} $P(M,b)$ coincides with the compact-open topology, so that $P(M,b)$ is contractible.

The guiding philosophy of our work is that a connected symmetric space actually is a Banach homogeneous space: It can be identified with the quotient of its automorphism group by a stabilizer subgroup (cf.\ \cite{Klo09b}). This is based on the theorem that its automorphism group is a Banach--Lie group. Therefore, we recall symmetric Lie groups, symmetric pairs and the functor $\Sym$ that assigns to a symmetric pair its quotient symmetric space (cf.\ \cite{Klo10Subspaces}). Since this functor preserves the exactness of short exact sequences, it constitutes a useful tool to translate between symmetric Lie groups and symmetric spaces.

\tableofcontents

%
%
%
\section{Basic Concepts and Notation}
%
%
\subsection{Terminology for Submanifolds and Lie Subgroups}
\label{sec:terminologySub}
A subset $N$ of a smooth Banach manifold $M$ is called a \emph{local submanifold at $x\in N$} if there exists a chart $\varphi\colon U\rightarrow V\subseteq E$ of $M$ at $x$ and a closed subspace $F$ of $E$ such that $\varphi(U\cap N)=V\cap F$. If $N$ is a local submanifold at each $x\in N$, then it is called a \emph{submanifold} and we obtain charts $\varphi|_{U\cap N}^{V\cap F}$ for $N$ that define on $N$ the structure of a manifold, which is compatible with the subspace topology. If each $F$ can be chosen as a split subspace of $E$, then $N$ is called a \emph{split submanifold}.
Note that, for each $x\in N$, the tangent map $T_x\iota$ of the inclusion map $\iota\colon N\hookrightarrow M$ is a (closed) topological embedding. A submanifold $N$ splits if and only if the inclusion map $\iota$ is an immersion, i.e., if for each $x\in N$, the image of the topological embedding $T_x\iota$ splits as a Banach space.

A subgroup $H$ of a Banach--Lie group $G$ is called a \emph{Lie subgroup} if it is a submanifold of $G$. It is called a \emph{split Lie subgroup} if it is even a split submanifold.
Every Lie subgroup is closed. Given a Lie subgroup $H\leq G$ with inclusion map $\iota\colon H\hookrightarrow G$, the Lie algebra $L(H)$ of $H$ can be identified with $\mathfrak{h}:=L(\iota)(L(H))$, which is given by $\mathfrak{h}=\{x\in L(G)\colon \exp_G(\RR x)\subseteq H\}$. There exists an open $0$-neighborhood $V\subseteq L(G)$ such that $\exp_G|_V$ is a local diffeomorphism onto an open subset of $M$ and $\exp_G(V\cap \mathfrak{h})=\exp_G(V)\cap H$ (cf.\ \cite[Th.~IV.3.3]{Nee06} or \cite[Prop.~3.4]{Hof75}).

An \emph{integral subgroup} of $G$ is a subgroup $H\leq G$ endowed with a Banach--Lie group structure such that the inclusion map $\iota\colon H\rightarrow G$ is smooth and $L(\iota)$ is a topological embedding.
Given some closed subalgebra $\mathfrak{h}$ of the Lie algebra $L(G)$ of $G$, then the subgroup $H:=\langle\exp_G(\mathfrak{h})\rangle$ carries a unique structure of a connected integral subgroup of $G$ with Lie algebra $\mathfrak{h}$ (cf.\ \cite[Satz~12.3]{Mai62}).
More generally, every Lie algebra $\mathfrak{h}$ that admits an injective smooth homomorphism $\mathfrak{h}\rightarrow L(G)$ is integrable (cf.\ \cite[(***) in \S 3]{EK64} or \cite[Cor.~IV.4.10]{Nee06}).
Note that an integral subgroup $H\leq G$ that is compatible with the subspace topology is a Lie subgroup.

%
%
\subsection{Lie Triple Systems and Symmetric Spaces} \label{sec:LtsAndSymSpace}
A \emph{Lie triple system} (cf.\ \cite{Loo69}) is a Banach space $\mathfrak{m}$ endowed with a continuous trilinear map\linebreak $[\cdot,\cdot,\cdot]\colon \mathfrak{m}^3\rightarrow\mathfrak{m}$ that satisfies $[x,x,y]=0$ and $[x,y,z] + [y,z,x] + [z,x,y] = 0$ as well as
\[[x,y,[u,v,w]] \ =\ [[x,y,u],v,w]+[u,[x,y,v],w]+[u,v,[x,y,w]]\]
for all $x,y,z,u,v,w\in \mathfrak{m}$. Continuous linear maps between Lie triple systems that are compatible with the Lie brackets are called \emph{morphisms}. A subspace $\mathfrak{n}$ of $\mathfrak{m}$ is called a \emph{triple subsystem} (denoted by $\mathfrak{n}\leq \mathfrak{m}$) if it is stable under the Lie bracket. If it satisfies the stronger condition $[\mathfrak{n},\mathfrak{m},\mathfrak{m}]\subseteq \mathfrak{n}$, then it is called an \emph{ideal} and we write $\mathfrak{n} \unlhd \mathfrak{m}$. An ideal $\mathfrak{n}\unlhd\mathfrak{m}$ automatically satisfies also the conditions $[\mathfrak{m},\mathfrak{n},\mathfrak{m}]\subseteq \mathfrak{n}$ and $[\mathfrak{m},\mathfrak{m},\mathfrak{n}]\subseteq \mathfrak{n}$.
The closed ideal $\mathfrak{z}(\mathfrak{m}):=\cap_{y,z\in\mathfrak{m}}\ker([\cdot,y,z])\unlhd \mathfrak{m}$ is called the \emph{center of $\mathfrak{m}$}. It also satisfies $[\mathfrak{m},\mathfrak{z}(\mathfrak{m}),\mathfrak{m}]=0$ and $[\mathfrak{m},\mathfrak{m},\mathfrak{z}(\mathfrak{m})]=0$.

A \emph{reflection space} (cf.\ \cite{Loo67a,Loo67b}) is a set $M$ endowed with a multiplication map $\mu\colon M\times M\rightarrow M$, $(x,y)\mapsto x\cdot y$, such that each left multiplication map $\mu_x:=\mu(x,\cdot)$ (with $x\in M$) is an involutive automorphism of $(M,\mu)$ with fixed point $x$. A subset of $M$ that is stable under $\mu$ is called a \emph{reflection subspace}. Given a subset $S\subseteq M$, we denote by $\langle S\rangle\leq M$ the \emph{generated reflection subspace}, i.e., the smallest reflection subspace of $M$ that contains $S$.
A \emph{symmetric space} is a smooth Banach manifold $M$ endowed with a smooth multiplication map $\mu\colon M\times M\rightarrow M$ such that $(M,\mu)$ is a reflection space for which each $x\in M$ is an isolated fixed point of the symmetry $\mu_x$. Maps between reflection spaces that are compatible with multiplication are called \emph{homomorphisms}. Concerning symmetric spaces, we refer to smooth homomorphisms as \emph{morphisms}.
If there is no confusion, we usually denote a reflection space (resp., symmetric space) simply by $M$ instead by $(M,\mu)$.

The following facts about symmetric spaces can be essentially found in \cite{Klo09b}, which is partially due to \cite{Ber08} and \cite{Nee02Cartan}, and are presented in nearly this form also in \cite{Klo10Subspaces}.
The tangent bundle $TM$ endowed with the multiplication $T\mu$ is a symmetric space. In each tangent space $T_xM$ (with $x\in M$), the product satisfies $v\cdot w =2v - w$. A smooth vector field on $M$ is called a \emph{derivation} if it is a morphism of symmetric spaces. Note that every derivation is a complete vector field. The set $\Der(M)$ of all derivations is a Lie subalgebra of the Lie algebra of all smooth vector fields on $M$.\footnote{Here, the term \emph{Lie algebra} does not include a topological structure.} Given a distinguished point $b\in M$, called the \emph{base point}, the symmetry $\mu_b$ induces an involutive automorphism $(\mu_b)_\ast$ of $\Der(M)$ given by $(\mu_b)_\ast(\xi):=T\mu_b\circ\xi\circ\mu_b$. The (+1)-eigenspace $\Der(M)_+$ of $(\mu_b)_\ast$ is a subalgebra of $\Der(M)$ and coincides with the kernel of the evaluation map $\ev_b\colon\Der(M)\rightarrow T_bM$, $\xi\mapsto \xi(b)$. The (-1)-eigenspace $\Der(M)_-$ of $(\mu_b)_\ast$ is stable under the triple bracket $[\cdot,\cdot,\cdot]:=[[\cdot,\cdot],\cdot]$. Via the evaluation isomorphism $\ev_b|_{\Der(M)_-}\colon\Der(M)_-\rightarrow T_bM$ of vector spaces, the tangent space $T_bM$ can be equipped with that triple bracket. It becomes a \emph{Lie triple system} that we denote by $\Lts(M,b)$. Assigning to each morphism of pointed symmetric spaces its tangent map at the base point, we obtain a covariant functor $\Lts$ (called the \emph{Lie functor}) from the category of pointed symmetric spaces to the category of Lie triple systems.

A Banach space can be considered as a symmetric space with natural multiplication $x\cdot y :=2x-y$. From this perspective, a smooth curve $\alpha\colon \RR\rightarrow M$ is called a \emph{one-parameter subspace of $M$} if $\alpha$ is a morphism of symmetric spaces. For each $v\in \Lts(M,b)$, there is a unique one-parameter subspace $\alpha_v$ with $\alpha_v^\prime(0)=v$. The map
$\Exp_{(M,b)}\colon \Lts(M,b)\rightarrow M$, $v\mapsto \alpha_v(1)$
is called the \emph{exponential map of $(M,b)$}. It is a smooth map with $T_0\Exp_{(M,b)}=\id_{\Lts(M,b)}$, so that it is a local diffeomorphism at $0$ and hence admits restrictions that are charts at $b$ (called \emph{normal charts}).
A morphism $f\colon (M_1,b_1)\rightarrow (M_2,b_2)$ of pointed symmetric spaces intertwines the exponential maps in the sense that
$f\circ \Exp_{(M_1,b_1)} = \Exp_{(M_2,b_2)} \circ \Lts(f)$. For details concerning the exponential map, see \cite{Klo09b}, whose approach is based on affine connections.\footnote{In \cite{Klo09b}, one-parameter subspaces are not dealt with, but it is easy to check that they coincide with geodesics. Cf.\ \cite[p.~87]{Loo69} for the finite-dimensional case.}

Given a one-parameter subspace $\alpha$, we call the automorphisms $\tau_{\alpha,s}:=\mu_{\alpha(\frac{1}{2}s)}\circ\mu_{\alpha(0)}$, $s\in\RR$, of $M$ \emph{translations along $\alpha$}. They satisfy $\tau_{\alpha,s}(\alpha(t))=\alpha(t+s)$ for all $t\in\RR$.
If $M$ is connected, then any two points can be joined by a sequence of one-parameter subspaces, since we have normal charts. Therefore, in view of the identities $\alpha(1)=(\tau_{\alpha,\frac{1}{n}})^n(\alpha(0))$ for all $n\in\NN$, it is easy to see that a connected $M$ is generated by each subset $U\subseteq M$ with non-empty interior. As a consequence, the \emph{basic connected component} $M_0$ of $(M,b)$ is generated by the image of the exponential map $\Exp_{(M,b)}$.

The automorphism group $\Aut(M)$ of a reflection space $M$ (resp., of a symmetric space) has two important (normal) subgroups:
The set of all symmetries $\mu_x$, $x\in M$, generates a subgroup which is denoted by $\Inn(M)$ and is called the \emph{group of inner automorphisms}. The subgroup $G(M)$ generated by all products $\mu_x\mu_y$, $x,y\in M$, is called the \emph{group of displacements} (cf.\ \cite[p.~64]{Loo69}).
For a connected symmetric space $M$, these groups of automorphisms act transitively on $M$, since there are translations along one-parameter subspaces.

Given a homomorphism $f\colon M_1\rightarrow M_2$, then for all $g_1\in\Inn (M_1)$, there exists a (not necessarily unique) $g_2\in\Inn(M_2)$ with
\begin{equation} \label{eqn:imageOfInnAut}
	f\circ g_1 \ =\ g_2 \circ f,
\end{equation}
because for a decomposition $g_1= \mu_{x_1}\mu_{x_2}\cdots\mu_{x_n}$, we can put $g_2:=\mu_{f(x_1)}\mu_{f(x_2)}\cdots\mu_{f(x_n)}$. As a consequence, a homomorphism $f\colon M_1\rightarrow M_2$ of reflection spaces that is locally smooth around some $b\in M_1$ is automatically smooth (and hence a morphism of symmetric spaces) if $M_1$ is $\Inn(M_1)$-transitive. Thus, given pointed symmetric space $(M_1,b_1)$ and $(M_2,b_2)$ with Lie triple systems $\mathfrak{m}_1$ and $\mathfrak{m}_2$, respectively, then a homomorphism $f\colon (M_1,b_1)\rightarrow (M_2,b_2)$ of pointed reflection spaces that satisfies
$f\circ \Exp_{(M_1,b_1)} = \Exp_{(M_2,b_2)} \circ A$
for some continuous linear map $A\colon \mathfrak{m}_1\rightarrow \mathfrak{m}_2$  is a morphism of pointed symmetric spaces with $\Lts(f)=A$.

Morphisms of pointed symmetric spaces to which the Lie functor assigns the same map coincide if the domain space is connected. If the domain space is moreover 1-connected (i.e., connected and simply connected), then every morphism of the Lie triple systems can be uniquely integrated to a morphism of pointed symmetric spaces. We shall refer to this fact as the \emph{Integrability Theorem} (cf.\ \cite[Th.~5.20]{Klo09b}). Note that isomorphisms are integrated to isomorphisms.

A symmetric space $M$ carries a torsionfree natural affine connection such that all symmetries $\mu_x$, $x\in M$, are affine automorphisms. The geodesics are just the one-parameter subspaces. We shall only briefly touch these aspects and refer to \cite{Klo09b} for details on affine Banach manifolds.

%
%
\subsection{Reflection Subspaces and Quotients of Symmetric Spaces}
\label{sec:subspacesAndQuotients}

An \emph{integral subspace} of a symmetric space $M$ is a reflection subspace $N\leq M$ endowed with a symmetric space structure such that the inclusion map $\iota\colon N\rightarrow M$ is smooth and for each $b\in N$, the induced morphism $\Lts_b(\iota)\colon \Lts(N,b)\rightarrow \Lts(M,b)$ of Lie triple systems is a (closed) topological embedding.
In the light of (\ref{eqn:imageOfInnAut}), we know for $\Inn(N)$-transitive (e.g.\ connected) $N$ that, for each $b_1,b_2\in N$, the map $\Lts_{b_1}(\iota)$ is a topological embedding if and only if $\Lts_{b_2}(\iota)$ is one.

We shall frequently identify $\Lts(N,b)$ with its image $\mathfrak{n}\subseteq \Lts(M,b)$ under $\Lts_b(\iota)$. Thus, the exponential map $\Exp_{(N,b)}$ of $(N,b)$ is the restriction $\Exp_{(M,b)}|_\mathfrak{n}$ of the exponential map $\Exp_{(M,b)}$ of $(M,b)$.
The basic connected component $N_0$ of $N$ is given by $N_0=\langle\Exp_{(M,b)}(\mathfrak{n})\rangle$ (cf.\ Section~\ref{sec:LtsAndSymSpace}).
\begin{proposition}[{cf.\ \cite[Prop.~3.14]{Klo10Subspaces}}] \label{prop:integralSubreflectionSpace}
	Let $(M,b)$ be a pointed symmetric space and $\mathfrak{n}$ a closed triple subsystem of $\Lts(M,b)$. Then $N:=\langle \Exp_{(M,b)}(\mathfrak{n})\rangle \leq M$ can be uniquely made an integral subspace of $M$ with $\Lts(N,b)=\mathfrak{n}$. Note that $N$ is connected.
\end{proposition}

A \emph{symmetric subspace} of a symmetric space $M$ is a reflection subspace $N\leq M$ that is a submanifold of $M$. It is clear that such $N$ itself is a symmetric space.
Let $\iota\colon N\rightarrow M$ be the inclusion map. Then, for each $b\in N$, the induced morphism $\Lts_b(\iota)$ is a topological embedding, $N$ being a submanifold. Therefore, a symmetric subspace is an integral subspace and we shall frequently identify $\Lts(N,b)$ with its image $\mathfrak{n}$ in $\Lts(M,b)$ under $\Lts_b(\iota)$.
It is given by $\mathfrak{n}= \{x\in\Lts(M,b)\colon \Exp_{(M,b)}(\RR x)\subseteq N\}$ (cf.\ \cite[Prop.~3.17]{Klo10Subspaces}).
Every open reflection subspace of a symmetric space is a symmetric subspace and is automatically closed (cf.\ \cite[Prop.~3.16]{Klo10Subspaces}). In particular, the connected components of a symmetric space are symmetric subspaces.
\begin{proposition}[{cf.\ \cite[Prop.~3.18]{Klo10Subspaces}}] \label{prop:expChartOfSubsymSpace}
	Let $(M,b)$ be a pointed symmetric space and $(N,b)$ be a symmetric subspace with Lie triple system $\mathfrak{n}\leq \Lts(M,b)$. Then there exists an open $0$-neighborhood $V\subseteq\Lts(M,b)$ such that $\Exp_{(M,b)}|_V$ is a diffeomorphism onto an open subset of $M$ and 
	$$\Exp_{(M,b)}(V\cap \mathfrak{n}) \ =\ \Exp_{(M,b)}(V) \cap N.$$
\end{proposition}
\begin{proposition}[{cf.\ \cite[Prop.~3.23]{Klo10Subspaces}}] \label{prop:integralSubSpace=subSymSpace}
	Let $(M,b)$ be a pointed symmetric space and $(N,b)\leq (M,b)$ be an $\Inn(N)$-transitive (e.g.\ connected) integral subspace with Lie triple system $\mathfrak{n}\leq \Lts(M,b)$ that splits as a Banach space. Let $F$ be a complement of $\mathfrak{n}$, i.e., $\Lts(M,b)=\mathfrak{n}\oplus F$.
	Then the following are equivalent:
	\begin{enumerate}
		\item[\rm (a)] $N$ is a symmetric subspace of $M$.
		\item[\rm (b)] There exists a $0$-neighborhood $W\subseteq F$ with $N\cap \Exp_{(M,b)}(W)=\{b\}$.
	\end{enumerate}
\end{proposition}
\begin{proposition}[{cf.\ \cite[Cor.~3.26]{Klo10Subspaces}}]\label{prop:kernelOfMorphismOfPointedSymSpaces}
	Let $f\colon (M_1,b_1)\rightarrow (M_2,b_2)$ be a morphism of pointed symmetric spaces. Then its kernel $\ker(f):=f^{-1}(b_2)$ is a closed symmetric subspace of $(M_1,b_1)$ with Lie triple system $\Lts(\ker(f))=\ker(\Lts(f))$.
\end{proposition}

An equivalence relation $R\subseteq M\times M$ on a reflection space $M$ that is a reflection subspace of $M\times M$ is called a \emph{congruence relation}. A congruence relation $R$ is just an equivalence relation for which the multiplication map on $M$ induces a multiplication map on the quotient $M/R$ (cf.\ quotient laws for magmas in \cite[I.1.6]{Bou89Algebra}). It can be easily checked that $M/R$ becomes a reflection space.
Given a congruence relation $R$ on $M$, every inner automorphism $g\in\Inn(M)$ maps equivalence classes onto equivalence classes (cf.\ \cite[Sec.~3.5]{Klo10Subspaces}).
Thus, if the group $\Inn(M)$ of inner automorphisms of $M$ acts transitively on $M$ (e.g.\ if $M$ is a connected symmetric space), then a congruence relation is determined by each of its equivalence classes, which we then call \emph{normal reflection subspaces of $M$}, denoted by $N\unlhd M$. We denote the pointed reflection space $(M/R,N)$ also by $M/N$.

The Lie triple system $\mathfrak{n}\leq \Lts(M,b)$ of a closed normal symmetric subspace $N\unlhd M$ is a closed ideal (cf.\ \cite[Prop.~3.32]{Klo10Subspaces}), and conversely, the connected integral subspace $N:=\langle\Exp_{(M,b)}(\mathfrak{n})\rangle\leq M$ corresponding to a closed ideal $\mathfrak{n}\unlhd \Lts(M,b)$ is normal (cf.\ \cite[Prop.~3.38]{Klo10Subspaces}).

\begin{theorem}[{cf.\ \cite[Th.~3.44 and Rem.~3.45]{Klo10Subspaces}}] \label{th:quotients}
	Let $(M,b)$ be a pointed connected symmetric space with Lie triple system $\mathfrak{m}$ and let $(N,b)\unlhd (M,b)$ be a closed connected normal symmetric subspace with Lie triple system $\mathfrak{n}\leq \mathfrak{m}$. Then $M/N$ can be made a symmetric space with Lie triple system $\mathfrak{m}/\mathfrak{n}$ such that the quotient map $\pi\colon (M,b)\rightarrow M/N$ is a morphism of pointed symmetric spaces and $\Lts(\pi)\colon \mathfrak{m}\rightarrow\mathfrak{m}/\mathfrak{n}$ is the natural quotient map.
\end{theorem}

%
%
\subsection{Symmetric Lie Algebras and Lie Groups}\label{sec:symLieAlgAndLieGrp}
A \emph{symmetric Lie algebra} is a Banach--Lie algebra $\mathfrak{g}$ endowed with an involutive automorphism $\theta$ of $\mathfrak{g}$, i.e., $\theta^2=\id_\mathfrak{g}$. Given two symmetric Lie algebras $(\mathfrak{g}_1,\theta_1)$ and $(\mathfrak{g}_2,\theta_2)$, a continuous Lie algebra homomorphism $A\colon \mathfrak{g}_1 \rightarrow \mathfrak{g}_2$ is called a \emph{morphism of symmetric Lie algebras} if it satisfies $A\circ\theta_1=\theta_2\circ A$.
The kernel of a morphism $A\colon (\mathfrak{g}_1,\theta_1)\rightarrow (\mathfrak{g}_2,\theta_2)$ satisfies $\theta_1(\ker(A))\subseteq\ker(A)$, so that $(\ker(A),\theta_1|_{\ker(A)}^{\ker(A)})$ is a symmetric Lie algebra.
A symmetric Lie algebra $(\mathfrak{g},\theta)$ decomposes as the direct sum of its $(\pm 1)$-eigenspaces denoted by $\mathfrak{g}=\mathfrak{g}_+\oplus \mathfrak{g}_-$. The $(+1)$-eigenspace is a subalgebra of $\mathfrak{g}$. The $(-1)$-eigenspace $\mathfrak{g}_-$ becomes a Lie triple system by defining $[x,y,z]:=[[x,y],z]$ (cf.\ \cite[Prop.~5.9]{Klo09b}).
Assigning to each symmetric Lie algebra $(\mathfrak{g},\theta)$ the Lie triple system $\mathfrak{g}_-$ and to each morphism of symmetric Lie algebras its restriction to the Lie triple systems, we obtain a covariant functor $\calLts$.

Conversely, every Lie triple system $\mathfrak{m}$ arises as the direct summand $\mathfrak{g}_-$ of a symmetric Lie algebra $(\mathfrak{g},\theta)$. Indeed, the direct sum  $S(\mathfrak{m}):=\mathfrak{h}\oplus \mathfrak{m}$ where 
$$\mathfrak{h}:=\overline{\spann\{[x,y,\cdot]\in\gl(\mathfrak{m})\colon x,y\in\mathfrak{m}\}}\leq \gl(\mathfrak{m})$$
is the closure of the linear span of all continuous linear maps $[x,y,\cdot]$ becomes a Banach--Lie algebra with the Lie bracket defined by
$[A,B] := AB - BA$, $[A,x] := -[x,A] := Ax$ and $[x,y] := [x,y,\cdot]$
for all $A,B \in \mathfrak{h}$ and $x,y\in \mathfrak{m}$. It can obviously be considered as a symmetric Lie algebra (called the \emph{standard embedding}) with $S(\mathfrak{m})_+=\mathfrak{h}=\overline{[\mathfrak{m},\mathfrak{m}]}$ and $S(\mathfrak{m})_-=\mathfrak{m}$. Cf.\ \cite[p.~79]{Loo69} for the finite-dimensional case.

The center $\mathfrak{z}(\mathfrak{g})$ of a symmetric Lie algebra $(\mathfrak{g},\theta)$ is $\theta$-invariant, so that it is itself a symmetric Lie algebra denoted by $\mathfrak{z}(\mathfrak{g},\theta)$. It decomposes as a direct sum $\mathfrak{z}(\mathfrak{g})_+\oplus \mathfrak{z}(\mathfrak{g})_-$ with $\mathfrak{z}(\mathfrak{g})_+\leq \mathfrak{g}_+$ and $\mathfrak{z}(\mathfrak{g})_-\leq \mathfrak{g}_-$. Furthermore, we have $\mathfrak{z}(\mathfrak{g})_-\subseteq \mathfrak{z}(\mathfrak{g}_-)$. It is easy to check that the center $\mathfrak{z}(S(\mathfrak{m}))$ of the standard embedding of a Lie triple system $\mathfrak{m}$ moreover coincides with the center $\mathfrak{z}(\mathfrak{m})$.

A \emph{symmetric Lie group} is a Banach--Lie group $G$ together with an involutive automorphism $\sigma$ of $G$. A \emph{morphism between symmetric Lie groups $(G_1,\sigma_1)$ and $(G_2,\sigma_2)$} is a smooth group homomorphism $f\colon G_1\rightarrow G_2$ such that $f\circ\sigma_1=\sigma_2\circ f$.
The Lie functor $L$ assigns to each $(G,\sigma)$ the \emph{symmetric Lie algebra} $L(G,\sigma):=(L(G),L(\sigma))$ and to each morphism $f$ the morphism $L(f)$ of symmetric Lie algebras.
The kernel $\ker(f)\unlhd G_1$ of a morphism $f$ is a Lie subgroup (cf.\ \cite[Th.~II.2]{GN03}) that satisfies $\sigma_1(\ker(f))\subseteq\ker(f)$, so that $(\ker(f),\sigma_1|_{\ker(f)}^{\ker(f)})$ is a symmetric Lie group. Its Lie algebra is $(\ker(L(f)),L(\sigma_1)|_{\ker(L(f))}^{\ker(L(f))})$.

Given a symmetric Lie group $(G,\sigma)$ with symmetric Lie algebra $(\mathfrak{g},\theta)$, the subgroup $G^\sigma:=\{g\in G\colon \sigma(g)=g\}$ of $\sigma$-fixed points is a split Lie subgroup with Lie algebra $\mathfrak{g}_+$ (cf.\ \cite[Ex.~3.9]{Nee02Cartan}). Open subgroups of $G^\sigma$ are given by subgroups $K\leq G^\sigma$ satisfying $(G^\sigma)_0\subseteq K \subseteq G^\sigma$, where $(G^\sigma)_0$ denotes the identity component. For such a subgroup $K$, we call $((G,\sigma),K)$ a \emph{symmetric pair} and shall rather write $(G,\sigma,K)$.\footnote{Note that other authors do not include the involution $\sigma$ in their definition, but require its existence (cf.\ \cite{Hel01}).}
The quotient space $\Sym(G,\sigma,K):=G/K$ carries the structure of a pointed symmetric space with multiplication
$$gK \cdot hK:=g\sigma(g)^{-1}\sigma(h)K$$
and base point $K$ such that the quotient map $q\colon G\rightarrow G/K$ is a submersion.
Note that $q$ is a principal bundle with structure group $K$ 
that acts on $G$ by right translations (cf.\ \cite[III.1.5--6]{Bou89LieGroups}).
When we consider the underlying symmetric space of $G/K$ that is not pointed, then we shall frequently write $\calU(G/K)$ to prevent confusion.
The Lie triple system $\Lts(G/K)$ can be identified with $\mathfrak{g}_-$ via the isomorphism $(T_{\eins}q)|_{\mathfrak{g}_-}$. Then the exponential map of $G/K$ is given by
\begin{equation}\label{eqn:Exp=q exp}
	\Exp_{G/K}:=q\circ\exp_G|_{\mathfrak{g}_-}\colon \mathfrak{g}_- \rightarrow G/K,
\end{equation}
where $\exp_G$ denotes the exponential map of the Lie group $G$. For further details, cf.\ \cite[Ex.~3.9]{Nee02Cartan}.

%
\begin{proposition} \label{prop:1-connectedQuotient}
	Let $(G,\sigma,(G^\sigma)_0)$ be a symmetric pair. If $G$ is 1-connected, then the quotient $G/(G^\sigma)_0$ is also 1-connected.
\end{proposition}
\begin{proof}
	Let $\gamma$ be any loop in $G/(G^{\sigma})_0$. The quotient map $q\colon G\rightarrow G/(G^{\sigma})_0$ being a (principal) fiber bundle, there is a path $\widetilde\gamma$ in $G$ that is a lift of $\gamma$. Its endpoints belong to $(G^{\sigma})_0$ and can be joined by a path, since $(G^{\sigma})_0$ is connected. In this way, we obtain a loop in $G$ whose projection in $G/(G^{\sigma})_0$ is homotopic to $\gamma$. The loop in $G$ being null-homotopic, its projection is also null-homotopic.
\end{proof}
A \emph{morphism between symmetric pairs $(G_1,\sigma_1,K_1)$ and $(G_2,\sigma_2,K_2)$} is a morphism \linebreak $f\colon (G_1,\sigma_1)\rightarrow (G_2,\sigma_2)$ of symmetric Lie groups satisfying $f(K_1)\subseteq K_2$. Note that $f(G_1^{\sigma_1})\subseteq G_2^{\sigma_2}$ and $f((G_1^{\sigma_1})_0)\subseteq (G_2^{\sigma_2})_0$ are always satisfied and that, moreover, we have $f((G_1^{\sigma_1})_0)= (G_2^{\sigma_2})_0$ if $L(f)((\mathfrak{g}_1)_+)=(\mathfrak{g}_2)_+$.\footnote{Indeed, we have $f((G_1^{\sigma_1})_0) \ =\ \langle f(\exp_{G_1}((\mathfrak{g}_1)_+))\rangle \ =\ \langle\exp_{G_2}((\mathfrak{g}_2)_+) \rangle \ =\ (G_2^{\sigma_2})_0.$}\label{page:L(f)(g1_+)=g2_+}
Every morphism $f$ induces a unique map
$$\Sym(f):=f_\ast\colon G_1/K_1 \rightarrow G_2/K_2$$
with $\Sym(f)\circ q_1 = q_2\circ f$ that is automatically a morphism of pointed symmetric spaces.
The assignment $\Sym$ is a covariant functor from the category of symmetric pairs to the category of pointed symmetric spaces. Denoting by $F$ the forgetful functor from the category of symmetric pairs to the category of symmetric Lie groups, we have
\begin{equation} \label{eqn:LtsSym=calLtsL}
	\Lts\circ\Sym \ =\ \calLts\circ L \circ F,
\end{equation}
where we read $L$ as the Lie functor applied to symmetric Lie groups (cf.\ \cite[Sec.~2.3]{Klo10Subspaces}).

%
\begin{proposition}[The functor $\Sym$ applied to surjective and injective morphisms]
\label{prop:symToSurjectiveInjevtive}
	We have:
	\begin{enumerate}
		\item[\rm (1)] Given a surjective morphism $f\colon (G_1,\sigma_1,K_1)\rightarrow (G_2,\sigma_2,K_2)$ of symmetric pairs, the morphism $\Sym(f)\colon G_1/K_1\rightarrow G_2/K_2$ is also surjective.
		\item[\rm (2)] Let $f\colon (G_1,\sigma_1)\rightarrow (G_2,\sigma_2)$ be an injective morphism of a symmetric Lie group $(G_1,\sigma_1)$ to the underlying symmetric Lie group of a symmetric pair $(G_2,\sigma_2,K_2)$. Then\linebreak $(G_1,\sigma_1,K_1)$ with $K_1:=f^{-1}(K_2)$ is a symmetric pair turning $f$ into a morphism of symmetric pairs such that the morphism $\Sym(f)\colon G_1/K_1\rightarrow G_2/K_2$ is also injective.	(Note that $G_1^{\sigma_1}=f^{-1}(G_2^{\sigma_2})$.)
	\end{enumerate}
\end{proposition}
\begin{proof}
	The assertion (1) ist trivial and the assertion (2) is due to \cite[Lem.~2.1]{Klo10Subspaces}.
\end{proof}

%
%
\begin{example} \label{ex:integralSubspaceOfHomogeneousSpace}
	Let $(G,\sigma_G,K_G)$ be a symmetric pair and $\iota\colon (H,\sigma_H)\rightarrow (G,\sigma_G)$ be an injective morphism of symmetric Lie groups such that $\iota\colon H\rightarrow G$ is an integral subgroup.
	With $K_H:=\iota^{-1}(K_G)$, the map $\iota\colon (H,\sigma_H,K_H)\rightarrow (G,\sigma_G,K_G)$ is a morphism of symmetric pairs and $\Sym(\iota)\colon H/K_H\rightarrow G/K_G$ is an injective morphism of pointed symmetric spaces (cf.\ Proposition~\ref{prop:symToSurjectiveInjevtive}). The induced map $\Lts(\Sym(\iota))\colon \Lts(H/K_H)\rightarrow \Lts(G/K_G)$ is given by $\calLts(L(\iota))$ (cf.\ (\ref{eqn:LtsSym=calLtsL})), i.e., by the topological embedding 
	$\mathfrak{h}_- \hookrightarrow \mathfrak{g}_-$,
	where $\mathfrak{g}$ and $\mathfrak{h}$ denote the Lie algebras of $G$ and $H$, respectively.
	
	Therefore $\Sym(\iota)$ is an integral subspace if $H/K_H$ is $\Inn(\calU(H/K_H))$-transitive, but this additional assumption is not necessary (cf.\ \cite[Ex.~3.7]{Klo10Subspaces}).
	Furthermore, there is no other integral subspace structure on (the reflection subspace) $H/K_H$ with Lie triple system\linebreak $\mathfrak{h}_-\hookrightarrow \mathfrak{g}_-$ if we require, in addition, that the natural action $\tau\colon H \times H/K_H \rightarrow H/K_H$, $(g,hK_H)\mapsto ghK_H$ is smooth (cf.\ \cite[Rem.~3.9]{Klo10Subspaces}).
	(Actually, it suffices to require that it induces smooth maps $\tau_h\colon H/K_H \rightarrow H/K_H$ for all $h\in H$.)
\end{example}
\begin{definition} \label{def:exactSymPairMorphisms}
A sequence
$$(G_1,\sigma_1,K_1)\rightarrow (G_2,\sigma_2,K_2)\rightarrow \cdots \rightarrow (G_n,\sigma_n,K_n)$$
of morphisms of symmetric pairs is called \emph{exact}, if the sequences $G_1 \rightarrow G_2 \rightarrow\cdots\rightarrow G_n$ and $K_1 \rightarrow K_2 \rightarrow\cdots\rightarrow K_n$ of group homomorphisms are exact.
\end{definition}
\begin{remark}\label{rem:characterizeExactness}
	Given a sequence
	$$\eins \longrightarrow (G_1,\sigma_1,K_1)\stackrel{f_1}{\longrightarrow} (G_2,\sigma_2,K_2)\stackrel{f_2}{\longrightarrow} (G_3,\sigma_3,K_3) \longrightarrow \eins$$
	for which $\eins\rightarrow G_1\rightarrow G_2 \rightarrow G_3\rightarrow \eins$ is exact. Then $\eins\rightarrow K_1\rightarrow K_2 \rightarrow K_3\rightarrow \eins$ is exact if and only if $f_2(K_2)=K_3$ and $K_1=f_1^{-1}(K_2)$.
	
	Indeed, exactness at $K_2$ means that $f_1(K_1)=\ker(f_2)\cap K_2 \ (=\im(f_1) \cap K_2)$ and is hence equivalent to $K_1=f_1^{-1}(K_2)$ by the injectivity of $f_1$.
\end{remark}
\begin{example} \label{ex:typicalExactSequence}
	Starting with a surjective morphism $f\colon (G_2,\sigma_2,K_2)\rightarrow (G_3,\sigma_3,K_3)$ that satisfies $f(K_2)=K_3$, we automatically obtain an exact sequence
	$$\eins \rightarrow (\ker(f),\sigma_2|_{\ker(f)},K_2\cap \ker(f))\hookrightarrow (G_2,\sigma_2,K_2)\rightarrow (G_3,\sigma_3,K_3) \rightarrow \eins$$
	(cf.\ Proposition~\ref{prop:symToSurjectiveInjevtive}(2) to see that $K_2\cap \ker(f)$ turns $(\ker(f),\sigma_2|_{\ker(f)})$ actually into a symmetric pair).
\end{example}
For homomorphisms of pointed reflection spaces (resp., morphisms of pointed symmetric spaces), we define exact sequences as for pointed sets. We shall denote a singleton space simply by $o$. Note that a homomorphism $f\colon (M_1,b_1)\rightarrow (M_2,b_2)$ with trivial kernel $\ker(f):=f^{-1}(b_2)$ is not necessarily injective, but it is so if its domain space $(M_1,b_1)$ is $\Inn(M_1)$-transitive (cf.\ (\ref{eqn:imageOfInnAut})).
\begin{proposition}[Exactness of the functor $\Sym$]
\label{prop:exactnessOfSym}
	Let
	$$\eins\longrightarrow (G_1,\sigma_1,K_1)\stackrel{f_1}{\longrightarrow} (G_2,\sigma_2,K_2)\stackrel{f_2}{\longrightarrow} (G_3,\sigma_3,K_3)\longrightarrow \eins$$
	be an exact sequence of symmetric pairs. Then the sequence
	$$o \longrightarrow G_1/K_1\stackrel{(f_1)_\ast}{\longrightarrow} G_2/K_2\stackrel{(f_2)_\ast}{\longrightarrow} G_3/K_3\longrightarrow o$$
	of pointed symmetric spaces is exact and $(f_1)_\ast$ is actually injective.
	If $f_1$ is a topological embedding, then $(f_1)_\ast$ also is a topological embedding.
\end{proposition}
\begin{proof}
	By Remark~\ref{rem:characterizeExactness}, we have $K_1=f_1^{-1}(K_2)$, so that the injectivity of $(f_1)_\ast$ and the surjectivity of $(f_2)_\ast$ follow by Proposition~\ref{prop:symToSurjectiveInjevtive}.
	The assertion $\im((f_1)_\ast)\subseteq \ker((f_2)_\ast)$ follows immediately from $\im(f_1)\subseteq \ker(f_2)$. To see that $\ker((f_2)_\ast) \subseteq \im((f_1)_\ast)$, we consider any $gK_2\in G_2/K_2$ satisfying $(f_2)_\ast(gK_2)=K_3$, i.e., $f_2(g)\in K_3$, and shall show that\linebreak $gK_2\in\im((f_1)_\ast)$. By the surjectivity of $K_2\rightarrow K_3$, there is a $g^\prime\in K_2$ with $f_2(g^\prime)=f_2(g)$, so that $f_2(g(g^\prime)^{-1})=\eins$. By the exactness of $G_1\rightarrow G_2\rightarrow G_3$, this entails the existence of some $g^{\prime\prime}\in G_1$ with $f_1(g^{\prime\prime})=g(g^\prime)^{-1}$. Hence, we obtain $(f_1)_\ast(g^{\prime\prime}K_1)=g(g^\prime)^{-1}K_2=gK_2$.
	
	Assume now that $f_1$ is a topological embedding.
	We consider $G_1$ as a subgroup of $G_2$ and note that it is a Lie subgroup, since it is the kernel of $f_2$.
	By Example~\ref{ex:integralSubspaceOfHomogeneousSpace}, the morphism $f_1$ induces an integral subspace $(f_1)_\ast=\Sym(f_1)\colon G_1/K_1\rightarrow G_2/K_2$. On the other hand, $N:=\im((f_1)_\ast)=\ker((f_2)_\ast)$ (cf.\ Proposition~\ref{prop:exactnessOfSym}) is a symmetric subspace of $G_2/K_2$ (cf.\ Proposition~\ref{prop:kernelOfMorphismOfPointedSymSpaces}). The smooth action $\tau\colon G_2 \times G_2/K_2 \rightarrow G_2/K_2$, $(g,hK_2)\mapsto ghK_2$ restricts to a smooth action $G_1 \times N \rightarrow N$, since, for all $g\in G_1$ and $hK_2\in N$, we have
	$$(f_2)_\ast(\tau(g,hK_2)) \ =\ f_2(gh)K_2 \ =\ f_2(g)f_2(h)K_2 \ =\ f_2(h)K_2 \ =\ (f_2)_\ast(hK_2) \ =\ K_3.$$
	Therefore, by Example~\ref{ex:integralSubspaceOfHomogeneousSpace}, the manifold structure on $N$ coincides with the one given by the integral subspace $(f_1)_\ast$, so that $(f_1)_\ast$ is a topological embedding.
\end{proof}
%
%

%
%
\subsection{The Automorphism Group of a Connected Symmetric Space}
\label{sec:AutOfConnectedSymSpace}
Let $M$ be a connected symmetric space. The Lie algebra $\Der(M)$ of derivations
can be uniquely turned into a Banach--Lie algebra such that for each frame $p$ in the frame bundle $\Fr(M)$ over $M$, the map
$$\Der(M)\rightarrow T_p(\Fr(M)),\ \xi\mapsto \left.\frac{d}{dt}\right|_{t=0}\Fr(\flow^{\xi}_t)(p)$$
is an embedding of Banach spaces, where $\flow^{\xi}_t$ is the time-$t$-flow of $\xi$ and $\Fr(\flow^{\xi}_t)$ is its induced automorphism of the frame bundle (cf.\ \cite[Cor.~5.18]{Klo09b}).
With regard to the natural affine connection on $M$, the Banach space structure can also be obtained by the requirement that 
$$\Der(M)\rightarrow T_bM \times \gl(T_bM),\ \xi\mapsto \big(\xi(b),\ v\mapsto\nabla\!_v\xi\big)$$
is an embedding of Banach spaces (cf.\ \cite[Prop.~2.2 and Prop.~5.17]{Klo09b}).
Given a base point $b\in M$, the involutive automorphism $(\mu_b)_\ast$ of the Lie algebra $\Der(M)$ is actually continuous and hence an automorphism of the Banach--Lie algebra $\Der(M)$, so that $(\Der(M),(\mu_b)_\ast)$ is a symmetric (Banach--)Lie algebra  and the evaluation map $\ev_b\colon\Der(M)_-\rightarrow \Lts(M,b)$ is an isomorphism of Lie triple systems (cf.\ \cite[Prop.~5.23]{Klo09b}).

The automorphism group $\Aut(M)$ can be turned into a Banach--Lie group such that
$$\exp\colon \Der(M)\rightarrow \Aut(M),\ \xi \mapsto \flow^{-\xi}_1$$
is its exponential map. The natural map $\tau\colon \Aut(M)\times M \rightarrow M$ is a transitive smooth action.
Together with the conjugation map $c_{\mu_b}\colon \Aut(M)\rightarrow \Aut(M)$, $g\mapsto\mu_b\circ g \circ \mu_b$, the automorphism group $\Aut(M)$ becomes a symmetric Lie group with Lie algebra $L(\Aut(M),c_{\mu_b})=(\Der(M),(\mu_b)_\ast)$. The stabilizer subgroup $\Aut(M)_b\leq \Aut(M)$ leads to the symmetric pair $(\Aut(M),c_{\mu_b},\Aut(M)_b)$. The induced symmetric space $\Aut(M)/\Aut(M)_b$ is isomorphic to $M$ via the isomorphism $\Phi\colon \Aut(M)/\Aut(M)_b\rightarrow M$ given by $\Phi(g\Aut(M)_b):=g(b)$. We refer to this fact as the \emph{homogeneity of connected symmetric spaces}. For further details, see \cite[Sec.~5.5]{Klo09b}.

For some purposes, the automorphism group of a connected symmetric space $M$ is too large and it is useful to consider the group $G(M)$ of displacements. In the finite-dimensional setting, this group is an integral subgroup of $\Aut(M)$ and leads to a further identification $M\cong G(M)/G(M)_b$ (cf.\ \cite{Loo69}). 
In \cite[Sec.~3.4]{Klo10Subspaces}, the author deals with the Banach case:

Let $(M,b)$ be a pointed connected symmetric space. Considering the symmetric Lie group $(\Aut(M),c_{\mu_b})$ with Lie algebra $(\Der(M),(\mu_b)_\ast)$, we denote by $G^\prime(M,b):=\langle\exp(\mathfrak{g}^\prime(M,b))\rangle$ the connected integral subgroup of $\Aut(M)$ that belongs to the closed subalgebra $$\mathfrak{g}^\prime(M,b):=\overline{[\Der(M)_-,\Der(M)_-]}\oplus\Der(M)_-$$
of $\Der(M)$. 
It satisfies $G(M)\leq G^\prime(M,b)\leq \overline{G(M)}$, where the closure is taken in $\Aut(M)$ (cf.\ \cite[Rem.~3.30]{Klo10Subspaces}).
Further, $G^\prime(M,b)$ is $c_{\mu_b}$-invariant and $(G^\prime(M,b),\sigma^\prime)$ with $\sigma^\prime:=c_{\mu_b}|_{G^\prime(M,b)}$ is a symmetric Lie group.
The symmetric pair $(G^\prime(M,b),\sigma^\prime,G^\prime(M,b)_b)$ where $G^\prime(M,b)_b\leq G^\prime(M,b)$ denotes the stabilizer subgroup of $b$ leads to an isomorphism
$$\Phi^\prime\colon G^\prime(M,b)/G^\prime(M,b)_b\rightarrow M,\,  gG^\prime(M,b)_b\mapsto g(b)$$
of symmetric spaces (cf.\ \cite[Prop.~3.28]{Klo10Subspaces}).
%
%
%
%
%
%
%
\section{Integrability of Lie Triple Systems}
The heart of this section is a characterization of integrable Lie triple systems. Along the way to this goal, we provide material concerning path and loop spaces and universal covering spaces of symmetric spaces. 
%
%
%
\subsection{Universal Covering Morphisms}
\begin{proposition}
\label{prop:universalCoveringMorphismOfSymSpace}
	Let $M$ be a connected symmetric space and $q_M\colon \widetilde M \rightarrow M$ a universal covering map of topological spaces. Then $\widetilde M$ carries a unique structure of a symmetric space for which $q_M$ is a morphism of symmetric spaces.
\end{proposition}
\begin{definition}
	We call $q_M$ a \emph{universal covering morphism (over $M$)}.
\end{definition}
\begin{proof}
	By \cite[Ex.~3.2Q(vii)]{AMR88}, $\widetilde M$ can be uniquely turned into a smooth Banach manifold making $q_M$ a local diffeomorphism. Fixing base points $\widetilde b\in\widetilde M$ and $b:=q_M(\widetilde b)\in M$ , the multiplication map $\mu$ on $M$ can be uniquely lifted to a continuous map $\widetilde\mu\colon \widetilde M \times \widetilde M \rightarrow \widetilde M$ with $\widetilde\mu(\widetilde b, \widetilde b)=\widetilde b$ by the lifting property of covering maps, i.e., we have $q_M\circ \widetilde\mu = \mu\circ (q_M\times q_M)$. Since $q_M$ and $q_M\times q_M$ are local diffeomorphisms, $\widetilde\mu$ is automatically smooth. It suffices to show that $(\widetilde M,\widetilde\mu)$ is a symmetric space.
	
	For this, we must verify the identities $x\cdot x = x$, $x \cdot (x\cdot y) = y$ and $x\cdot (y\cdot z) = (x\cdot y)\cdot (x\cdot z)$ for all $x,y,z\in\widetilde M$ and must check that for each $x\in\widetilde M$, the fixed point $x$ of $\widetilde\mu_x$ is an isolated fixed point. Exemplarily, we present the verification of $x \cdot (x\cdot y) = y$: We observe that it holds for $x:=y:=\widetilde b$ and that
	$$q_M(x \cdot (x \cdot y)) \ =\ q_M(x) \cdot (q_M(x) \cdot q_M(y)) \ =\ q_M(y),$$
	so that it holds for all $x$ and $y$ by uniqueness of liftings.
	Given $x\in\widetilde M$, we have $q_M\circ \widetilde\mu_x = \mu_{q_M(x)}\circ q_M$. Since $q_M$ is a local diffeomorphism (around $x$) and since $q_M(x)$ is an isolated fixed point of $\mu_{q_M(x)}$, it is easy to see that $x$ is an isolated fixed point of $\widetilde\mu_x$.
\end{proof}
\begin{remark}\label{rem:Lts(covering)=Iso}
	A universal covering morphism $q\colon (\widetilde M,\widetilde b)\rightarrow (M,b)$ of pointed symmetric spaces being a local diffeomorphism, it induces an isomorphism $\Lts(q)$ of Lie triple systems. Thus, we shall frequently identify $\Lts(\widetilde M,\widetilde b)$ with $\Lts(M,b)$.
	Conversely, given another morphism $q^\prime\colon (M^\prime,b^\prime)\rightarrow (M,b)$ over $(M,b)$ with 1-connected $M^\prime$ such that $\Lts(q^\prime)$ is an isomorphism, then $q^\prime$ is a universal covering morphism, since $(M^\prime, b^\prime)$ and $(\widetilde M,\widetilde b)$ are naturally isomorphic over $(M,b)$ by the Integrability Theorem.
\end{remark}
\begin{proposition}
\label{prop:universalCoveringMorphismOfSymLieGroup}
	Let $(G,\sigma)$ be a connected symmetric Lie group and $q\colon \widetilde G \rightarrow G$ a universal covering morphism of Lie groups. Then there is a unique involutive automorphism $\widetilde\sigma\colon \widetilde G\rightarrow \widetilde G$ making $q\colon(\widetilde G,\widetilde \sigma)\rightarrow (G,\sigma)$ a morphism of symmetric Lie groups.
\end{proposition}
\begin{definition}
	We call $q$ a \emph{universal covering morphism over $(G,\sigma)$} and shall frequently denote it by $q_{(G,\sigma)}$.
\end{definition}
\begin{proof}
	By the lifting property of covering morphisms of Lie groups, we can lift the smooth group homomorphism $\sigma\circ q\colon \widetilde G \rightarrow G$ to a unique smooth group homomorphism $\widetilde\sigma\colon \widetilde G\rightarrow \widetilde G$ with $q\circ\widetilde\sigma = \sigma\circ q$. Since $\widetilde\sigma^2$ then satisfies
	$q\circ\widetilde\sigma^2 = \sigma^2\circ q=\id_G\circ q$, it is a lift of $q$ and is hence the identity map $\id_{\widetilde G}$ by uniqueness of liftings. Therefore, $q\colon(\widetilde G,\widetilde \sigma)\rightarrow (G,\sigma)$ is a morphism of symmetric Lie groups.
\end{proof}
\begin{remark} \label{rem:L(covering)=Iso}
	As in Remark~\ref{rem:Lts(covering)=Iso}, the morphism $q_{(G,\sigma)}$ induces an isomorphism $L(q_{(G,\sigma)})$ of symmetric Lie algebras, so that we shall frequently identify $L(\widetilde G,\widetilde \sigma)$ with $L(G,\sigma)$. Conversely, given any morphism $q^\prime\colon (G^\prime,\sigma^\prime)\rightarrow (G,\sigma)$ over $(G,\sigma)$ with 1-connected $G^\prime$ such that $L(q^\prime)$ is an isomorphism, then $q^\prime$ is a universal covering morphism.	
\end{remark}
%
%
%
%
\subsection{Path and Loop Spaces}
Let $(G,\sigma)$ be a symmetric Lie group with Lie algebra $(\mathfrak{g},\theta)$. The path group
$$P(G):=\{\gamma\in C([0,1],G)\colon \gamma(0)=\eins\}$$
of $G$, where the multiplication on $P(G)$ is defined pointwise, is a contractible (and hence 1-connected) Banach--Lie group with Lie algebra
$P(\mathfrak{g}):=\{\gamma\in C([0,1],\mathfrak{g})\colon \gamma(0)=0\}$
and exponential map
$P(\exp_G)\colon P(\mathfrak{g})\rightarrow P(G),\ \gamma\mapsto \exp_G\circ\gamma$
(cf.\ \cite{GN03}).
Endowing $P(G)$ with the involutive automorphism $P(\sigma)$ of $P(G)$ given by $P(\sigma)(\gamma):=\sigma\circ\gamma$, we make it a symmetric Lie group called the \emph{path group of $(G,\sigma)$} and denoted by $P(G,\sigma)$. Its Lie algebra is $P(\mathfrak{g})$ endowed with the involutive automorphism $P(\theta)$ of $P(\mathfrak{g})$ that is given by $P(\theta)(\gamma):=\theta\circ\gamma$. It is denoted by $P(\mathfrak{g},\theta)$.

The kernel $\Omega(G)\unlhd P(G)$ of the evaluation morphism
$$(\ev_1:=)\ \ev_1^{(G,\sigma)}\colon P(G,\sigma)\rightarrow (G,\sigma),\ \gamma\mapsto \gamma(1)$$
endowed with the involutive automorphism $\Omega(\sigma):=P(\sigma)|_{\Omega(G)}^{\Omega(G)}$ is called the \emph{loop group of $(G,\sigma)$} and is denoted by $\Omega(G,\sigma)$.
Its Lie algebra is the kernel $\Omega(\mathfrak{g})$ of the evaluation morphism
$$(\ev_1:=)\ \ev_1^{(\mathfrak{g},\theta)}\colon P(\mathfrak{g},\theta)\rightarrow (\mathfrak{g},\theta),\ \gamma\mapsto \gamma(1)$$
endowed with the involutive automorphism $\Omega(\theta):=P(\theta)|_{\Omega(\mathfrak{g})}^{\Omega(\mathfrak{g})}$. It is denoted by $\Omega(\mathfrak{g},\theta)$.

Let $(M,b)$ be a pointed symmetric space. The set
$$P(M,b):=\{\gamma\in C([0,1],M)\colon \gamma(0)=b\}$$
carries a natural structure of a pointed reflection space,
where the multiplication on $P(M,b)$ is defined pointwise and where the base point is given by the constant curve $\const_b$ with\linebreak value $b$. We intend to turn $P(M,b)$ into a pointed symmetric space with Lie triple system
$$P(\mathfrak{m}):=\{\gamma\in C([0,1],\mathfrak{m})\colon \gamma(0)=0\},$$
where the Lie bracket is defined pointwise.

\begin{proposition}\label{prop:P(G/K)}
	Let $(G,\sigma,K)$ be a symmetric pair with quotient morphism $q\colon G\rightarrow G/K$\linebreak and let $\pi\colon P(G)\rightarrow P(G)/P(G^\sigma)$ be the quotient morphism given by the symmetric pair $(P(G,\sigma),P(G^\sigma))$.
	Then the map $P(q)\colon P(G)\rightarrow P(G/K)$ defined by $P(q)(\gamma):=q\circ \gamma$ factors over $\pi$, i.e., there is a unique map $\Phi\colon P(G)/P(G^\sigma)\rightarrow P(G/K)$ with $\Phi\circ\pi=P(q)$. Furthermore, $\Phi$ is an isomorphism of pointed reflection spaces.
\end{proposition}
\begin{proof}
	Given any curves $\gamma_1, \gamma_2 \in P(G)$ with $\gamma_2=\gamma_1\delta$ for some $\delta\in P(G^\sigma)=P(K)$, it is easy to see that $P(q)(\gamma_2)= P(q)(\gamma_1)$, which shows that $P(q)$ factors over $\pi$.
	To see that the induced map $\Phi$ is injective, we take any curves $\gamma_1,\gamma_2\in P(G)$ with $P(q)(\gamma_1) = P(q)(\gamma_2)$. Then the curve $\delta:=\gamma_1^{-1}\gamma_2\in P(G)$ actually lies in $K$, i.e., is an element of $P(K)=P(G^\sigma)$, proving the injectivity of $\Phi$. Since $q$ is a fiber bundle, the map $P(q)$ is surjective by the path lifting property of fiber bundles, entailing the surjectivity of $\Phi$. Finally, we check that $\Phi$ is a homomorphism: Given any curves $\gamma_1,\gamma_2\in P(G)$, we have
	\begin{align*}
		\Phi(\pi(\gamma_1)\cdot\pi(\gamma_2))\
		&=\ \Phi\big(\pi(\gamma_1(\sigma\circ\gamma_1)^{-1}(\sigma\circ\gamma_2))\big) \ =\  q\circ (\gamma_1(\sigma\circ\gamma_1)^{-1}(\sigma\circ\gamma_2)) \\
		&=\ (q\circ\gamma_1)\cdot (q\circ\gamma_2) \ =\ \Phi(\pi(\gamma_1))\cdot \Phi(\pi(\gamma_2)).
		\qedhere
	\end{align*}
\end{proof}
\begin{corollary}\label{cor:P(M,b)isInn-transitive}
	For each pointed symmetric space $(M,b)$, the reflection space $P(M,b)$ is $\Inn(\calU(P(M,b)))$-transitive, where, to prevent confusion, $\calU(P(M,b))$ denotes the underlying reflection space of $P(M,b)$ that is not pointed.
\end{corollary}
\begin{proof}
	W.l.o.g., we assume $M$ to be connected. By homogeneity, it can be identified with a quotient as in the preceding proposition. Since $P(G)$ and hence $P(G)/P(G^\sigma)$ are connected, the symmetric space $P(G)/P(G^\sigma)$ is $\Inn(\calU(P(G)/P(G^\sigma)))$-transitive. Thus, also the reflection space $P(G/K)$ is $\Inn(\calU(P(G/K)))$-transitive.
\end{proof}
\begin{theorem}\label{th:pathSymSpace}
	\begin{enumerate}
		\item[\rm (1)]
		Let $(M,b)$ be a pointed symmetric space with Lie triple system $\mathfrak{m}$. The pointed reflection space $P(M,b)$
		carries a unique structure of a Banach manifold making it a pointed symmetric space with Lie triple system $P(\mathfrak{m})$ such that $$P(\Exp_{(M,b)})\colon P(\mathfrak{m})\rightarrow P(M,b),\ \gamma \mapsto \Exp_{(M,b)}\circ\gamma$$
		is its exponential map. This space $P(M,b)$ is 1-connected.
		\item[\rm (2)]
		With regard to (1), the map $\Phi$ of Proposition~\ref{prop:P(G/K)} becomes an isomorphism of pointed symmetric spaces. In particular, we have $\Phi\circ\Exp_{P(G)/P(G^\sigma)}=P(\Exp_{G/K})$.
	\end{enumerate}
\end{theorem}
\begin{definition}
	The pointed symmetric space $P(M,b)$ is called the \emph{path (symmetric) space of $(M,b)$}.
\end{definition}
\begin{proof}
	W.l.o.g., we assume $M$ to be connected and hence homogeneous. By Corollary~\ref{cor:P(M,b)isInn-transitive}, the reflection space $P(M,b)$ is $\Inn(\calU(P(M,b)))$-transitive. Therefore, it is clear that there exists at most one structure of a Banach manifold making $P(M,b)$ a pointed symmetric space with exponential map $P(\Exp_{(M,b)})$: Indeed, given two such structures, then for each $\gamma\in P(M,b)$, there exists some $g\in\Inn(\calU(P(M,b)))$ with $g(\const_b)=\gamma$, so that, considering a shareable\footnote{By a shareable chart, we mean a map that is a chart with respect to either manifold structure.} exponential chart $\varphi_{\const_b}=(P(\Exp_{(M,b)})|_V^U)^{-1}$ at $\const_b$, the map $\varphi_\gamma:=\varphi_{\const_b}\circ g^{-1}$ is a shareable chart at $\gamma$.
	
	With regard to the homogeneity of $M$, we consider the situation of (2), i.e., of Proposition~\ref{prop:P(G/K)}, and show that $P(G/K)$ can be turned into a pointed symmetric space with exponential map $P(\Exp_{G/K})$. It is clear that $P(G/K)$ inherits the structure of a pointed symmetric space via the map $\Phi$. Then the map $\Phi\circ\Exp_{P(G)/P(G^\sigma)}$ is its exponential map, so that we shall show that it is equal to $P(\Exp_{G/K})$. Indeed, we have
	\begin{eqnarray*}
		\Phi\circ\Exp_{P(G)/P(G^\sigma)} &=& \Phi\circ\pi \circ \exp_{P(G)}|_{P(\mathfrak{g})_-} \ =\ P(q) \circ P(\exp_{G})|_{P(\mathfrak{g}_-)} \\ &=& P(q\circ\exp_G|_{\mathfrak{g}_-}) \ =\ P(\Exp_{G/K}).
	\end{eqnarray*}
	By Proposition~\ref{prop:1-connectedQuotient}, the space $P(G)/P(G^\sigma)$ is 1-connected, whence $P(G/K)$ is 1-connected as well.
\end{proof}
\begin{proposition} \label{prop:compact-openP(M,b)}
	Let $(M,b)$ be a pointed symmetric space. Then the topology of its path symmetric space $P(M,b)$ coincides with the compact-open topology.
\end{proposition}
\begin{proof}
	We denote by $P(M,b)_{c.o.}$ the set of paths endowed with the compact-open topology. The spaces $P(M,b)$ and $P(M,b)_{c.o.}$ have the same neighborhood filter at $\const_b$, since the topology of $P(\mathfrak{m})$ is given also by the compact-open topology.
	Since the symmetric space $P(M,b)$ is connected and hence $\Inn(\calU(P(M,b)))$-transitive, it suffices to show that the symmetries around its points are also homeomorphisms of $P(M,b)_{c.o.}$. For this, it suffices to check that the multiplication on $P(M,b)$ is a continuous map $P(M,b)_{c.o.}\times P(M,b)_{c.o.}\rightarrow P(M,b)_{c.o.}$, but this is true, since
	it can be identified with the continuous map 
	$$P(\mu)\colon P(M\times M,(b,b))_{c.o.}\rightarrow P(M,b)_{c.o.},\ (\gamma_1,\gamma_2)\mapsto \mu\circ(\gamma_1,\gamma_2),$$
	where $\mu$ denotes the multiplication map on $M$.
\end{proof}
\begin{corollary}
	The path space $P(M,b)$ of a pointed symmetric space $(M,b)$ is contractible.
\end{corollary}
\begin{proof}
	Cf.\ \cite[VII.6, Prop.~6.18]{Bre93}. The homotopy $H\colon [0,1] \times P(M,b)\rightarrow P(M,b)$ can be defined as $H(s,\gamma)(t):=\gamma(st)$.
\end{proof}
Let $(M,b)$ be a pointed symmetric space with Lie triple system $\mathfrak{m}$. The evaluation map
$$(\ev_1:=)\ \ev_1^{(M,b)}\colon P(M,b)\rightarrow (M,b),\ \gamma \mapsto \gamma(1)$$
on its path space $P(M,b)$ is a homomorphism of pointed reflection spaces that satisfies
$$\ev_1^{(M,b)}\circ P(\Exp_{(M,b)}) \ =\ \Exp_{(M,b)}\circ \ev_1^\mathfrak{m}$$
with evaluation morphism
$$(\ev_1:=)\ \ev_1^\mathfrak{m}\colon P(\mathfrak{m})\rightarrow \mathfrak{m},\ \gamma\mapsto \gamma(1),$$
so that it is a morphism of pointed symmetric spaces with $\Lts(\ev_1^{(M,b)})=\ev_1^\mathfrak{m}$.
Its kernel $\Omega(M,b)\leq P(M,b)$ is a symmetric subspace called the \emph{loop (symmetric) space of $(M,b)$}, whose Lie triple system is the kernel $\Omega(\mathfrak{m})\unlhd P(\mathfrak{m})$ of $\ev_1^\mathfrak{m}$ (cf.\ Proposition~\ref{prop:kernelOfMorphismOfPointedSymSpaces}). The exponential map of $\Omega(M,b)$ is $\Omega(\Exp_{(M,b)}):=P(\Exp_{(M,b)})|_{\Omega(\mathfrak{m})}^{\Omega(M,b)}$. By Proposition~\ref{prop:compact-openP(M,b)}, the topology on $\Omega(M,b)$ is given by the compact-open topology. Therefore, for every pointed 1-connected symmetric space $(M,b)$, its loop space $\Omega(M,b)$ is connected.
%
%
%
\subsection{The Period Morphism of a Symmetric Lie Algebra}
\label{sec:periodMorphismOfSymLieAlgebra}
In this subsection, we collect useful definitions and results concerning the problem under which conditions a (symmetric) Lie algebra is \emph{integrable}, i.e., arises as the Lie algebra of a (symmetric) Lie group. In Section~\ref{sec:IntegrabilityCriterion}, we shall apply these considerations to solve the integration problem for Lie triple systems. It is a classical result that a Lie algebra $\mathfrak{g}$ is integrable if and only if its period group $\Pi(\mathfrak{g})$ is discrete. We refer to \cite{GN03}, where H.~Gl\"ockner and K.-H.~Neeb give a quite direct definition of the period group of a Lie algebra. For our purposes, we have to refine the process leading to this group by considering additional involutive automorphisms.
\begin{remark} \label{rem:integrabilityOfSymLieAlgebra}
	A symmetric Lie algebra $(\mathfrak{g},\theta)$ is integrable if and only if the underlying Lie algebra $\mathfrak{g}$ is integrable. Indeed, replacing a given Lie group $G$ with Lie algebra $\mathfrak{g}$ by its universal covering group, the involutive automorphism $\theta$ can be uniquely integrated to an involutive automorphism $\sigma$ of $G$.
\end{remark}
Let $(\mathfrak{g},\theta)$ be a symmetric Banach--Lie algebra. The quotient $\mathfrak{g}_{\ad}:=\mathfrak{g}/\mathfrak{z}(\mathfrak{g})$ endowed with the involutive automorphism $\theta_{\ad}$ that is induced by $\theta$ can be integrated to a 1-connected symmetric Lie group $(G_{\ad},\sigma_{\ad})$ (cf.\ \cite[p.~7]{GN03}).

The central extension
$$0\rightarrow \mathfrak{z}(\mathfrak{g},\theta)\hookrightarrow (\mathfrak{g},\theta)\rightarrow (\mathfrak{g}_{\ad},\theta_{\ad})\rightarrow 0,$$
can be pulled back via the evaluation morphism $\ev_1\colon P(\mathfrak{g}_{\ad},\theta_{\ad})\rightarrow (\mathfrak{g}_{\ad},\theta_{\ad})$ to a central extension
\begin{equation} \label{eqn:extensionHatP(g,Lsigma)}
	0\rightarrow \mathfrak{z}(\mathfrak{g},\theta)\hookrightarrow \widehat P(\mathfrak{g},\theta)\rightarrow P(\mathfrak{g}_{\ad},\theta_{\ad})\rightarrow 0
\end{equation}
where the symmetric Lie algebra $\widehat P(\mathfrak{g},\theta)$ is defined as the Lie algebra
$$\widehat P(\mathfrak{g}):=\{(\gamma,x)\in P(\mathfrak{g}_{\ad})\times\mathfrak{g}\colon\gamma(1)=x+\mathfrak{z}(\mathfrak{g})\}$$
endowed with the involutive automorphism $\widehat P(\theta) := \big(P(\theta_{\ad})\times \theta\big)|_{\widehat P(\mathfrak{g})}$.
Restricting (\ref{eqn:extensionHatP(g,Lsigma)}) to the preimage $\widehat\Omega(\mathfrak{g}):=\Omega(\mathfrak{g}_{\ad})\times\mathfrak{z}(\mathfrak{g})$ of $\Omega(\mathfrak{g}_{\ad})$, we obtain a central extension
\begin{equation} \label{eqn:extsenionHatOmega(g,Lsigma)}
	0\rightarrow \mathfrak{z}(\mathfrak{g},\theta)\hookrightarrow \widehat \Omega(\mathfrak{g},\theta)\rightarrow \Omega(\mathfrak{g}_{\ad},\theta_{\ad})\rightarrow 0,
\end{equation}
where $\widehat\Omega(\mathfrak{g},\theta)$ is the Lie algebra $\widehat\Omega(\mathfrak{g})$ endowed with the involutive automorphism
$$\widehat\Omega(\theta):=\widehat P(\theta)|_{\widehat \Omega(\mathfrak{g})}=\big(\Omega(\theta_{\ad})\times \theta\big)|_{\widehat \Omega(\mathfrak{g})}.$$

We claim that there exists an exact sequence
\begin{equation}\label{eqn:sequenceHatP(G,sigma)}
	0\rightarrow \mathfrak{z}(\mathfrak{g},\theta) \hookrightarrow \widehat P(G,\sigma)\stackrel{\chi_{(\mathfrak{g},\theta)}}{\longrightarrow} P(G_{\ad},\sigma_{\ad})\rightarrow \eins
\end{equation}
of symmetric Lie groups	corresponding to (\ref{eqn:extensionHatP(g,Lsigma)}), where $\widehat P(G,\sigma)$ is a suitable 1-connected Lie group $\widehat P(G)$ endowed with a suitable involutive automorphism $\widehat P(\sigma)$ and where the inclusion of $\mathfrak{z}(\mathfrak{g})$ into $\widehat P(G)$ is a topological embedding.\footnote{Following the notation of \cite{GN03}, we use notations like $\widehat P(G,\sigma)$ and (presently) $\widehat \Omega(G,\sigma)$ without assuming the existence of a symmetric Lie group $(G,\sigma)$ with Lie algebra $(\mathfrak{g},\theta)$.}
Indeed, there is a central group extension
$$0\rightarrow \mathfrak{z}(\mathfrak{g}) \hookrightarrow \widehat P(G)\rightarrow P(G_{\ad})\rightarrow \eins$$
(cf.\ \cite[p.~7]{GN03}, due to \cite{Est62}), so that we merely have to define $\widehat P(\sigma)$ as the unique integral of $\widehat P(\theta)$.
The kernel of the smooth group homomorphism
\begin{equation} \label{eqn:ev_1CircChi}
	\ev_1^{(G_{\ad},\sigma_{\ad})}\circ\chi_{(\mathfrak{g},\theta)}\colon \widehat P(G,\sigma)\rightarrow (G_{\ad},\sigma_{\ad}),\ g\mapsto \chi_{(\mathfrak{g},\theta)}(g)(1)
\end{equation}
is a symmetric Lie subgroup $\widehat\Omega(G,\sigma):=(\widehat\Omega(G),\widehat\Omega(\sigma)):= (\widehat\Omega(G),\widehat P(\sigma)|_{\widehat\Omega(G)})\unlhd \widehat P(G,\sigma)$,
whose Lie algebra is the kernel of the morphism
$$\ev_1^{(\mathfrak{g}_{\ad},\theta_{\ad})}\circ L(\chi_{(\mathfrak{g},\theta)})\colon \widehat P(\mathfrak{g},\theta)\rightarrow (\mathfrak{g}_{\ad},\theta_{\ad}),\ (\gamma,x) \mapsto \gamma(1)=x+\mathfrak{z}(\mathfrak{g})$$
of symmetric Lie algebras, i.e., the symmetric Lie algebra $\widehat\Omega(\mathfrak{g},\theta)$. Note that $\widehat P(G)/\widehat\Omega(G)\cong G_{\ad}$, so that $\widehat\Omega(G)$ is connected, $G_{\ad}$ being 1-connected (cf. \cite[p.~8]{GN03}).
Since we have $\widehat \Omega(G,\sigma)=\chi_{(\mathfrak{g},\theta)}^{-1}(\Omega(G_{\ad}))$, the exact sequence (\ref{eqn:sequenceHatP(G,sigma)}) can be restricted to the exact sequence
\begin{equation} \label{eqn:sequenceHatOmega(G,sigma)}
	0\rightarrow \mathfrak{z}(\mathfrak{g},\theta)\hookrightarrow \widehat \Omega(G,\sigma)\rightarrow \Omega(G_{\ad},\sigma_{\ad})\rightarrow \eins,
\end{equation}
which corresponds to (\ref{eqn:extsenionHatOmega(g,Lsigma)}).

The Lie algebra $\Omega(\mathfrak{g}_{\ad},\theta_{\ad})$ being integrable (to $\Omega(G_{\ad},\sigma_{\ad})$), there is a 1-connected symmetric Lie group $\widetilde\Omega(G_{\ad},\sigma_{\ad})$ whose Lie algebra is $\Omega(\mathfrak{g}_{\ad},\theta_{\ad})$.
We integrate the inclusion morphism
$$\Omega(\mathfrak{g}_{\ad},\theta_{\ad})\hookrightarrow \Omega(\mathfrak{g}_{\ad},\theta_{\ad})\times\mathfrak{z}(\mathfrak{g},\theta)=\widehat\Omega(\mathfrak{g},\theta),\ x\mapsto (x,0)$$
to a morphism
\begin{equation}\label{eqn:f_(g,theta)}
	f_{(\mathfrak{g},\theta)}\colon \widetilde\Omega(G_{\ad},\sigma_{\ad})\rightarrow \widehat\Omega(G,\sigma).
\end{equation}
Composing it with $\widehat\Omega(G,\sigma)\rightarrow \Omega(G_{\ad},\sigma_{\ad})$ (cf.\ (\ref{eqn:sequenceHatOmega(G,sigma)})) gives us a morphism
$$q_{\Omega(G_{\ad},\sigma_{\ad})}\colon \widetilde\Omega(G_{\ad},\sigma_{\ad})\stackrel{f_{(\mathfrak{g},\theta)}}{\longrightarrow} \widehat\Omega(G,\sigma) \rightarrow \Omega(G_{\ad},\sigma_{\ad})$$
satisfying $L(q_{\Omega(G_{\ad},\sigma_{\ad})})=\id_{\Omega(\mathfrak{g}_{\ad},\sigma_{\ad})}$, so that it is a universal covering morphism (cf.\ Remark~\ref{rem:L(covering)=Iso}). Since $\mathfrak{z}(\mathfrak{g})$ is the kernel of $\widehat\Omega(G,\sigma)\rightarrow \Omega(G_{\ad},\sigma_{\ad})$, the kernel $\pi_1(\Omega(G_{\ad}))$ of $q_{\Omega(G_{\ad},\sigma_{\ad})}$ is mapped by $f_{(\mathfrak{g},\theta)}$ onto $\im(f_{(\mathfrak{g},\theta)})\cap \mathfrak{z}(\mathfrak{g})$. Endowing the Lie group $\pi_1(\Omega(G_{\ad}))$ with the involutive automorphism $\pi_1(\Omega(\sigma_{\ad})):=\widetilde\Omega(\sigma_{\ad})|_{\pi_1(\Omega(G_{\ad}))}$ gives us a symmetric Lie group that we denote by $\pi_1(\Omega(G_{\ad},\sigma_{\ad}))$.
\begin{definition}
	The restriction $\per_{(\mathfrak{g},\theta)}\colon \pi_1(\Omega(G_{\ad},\sigma_{\ad}))\rightarrow \mathfrak{z}(\mathfrak{g},\theta)$ of $f_{(\mathfrak{g},\theta)}$ is called the \emph{period morphism of $(\mathfrak{g},\theta)$} and its image $\Pi(\mathfrak{g}):=\im(\per_{(\mathfrak{g},\theta)}) = \im(f_{(\mathfrak{g},\theta)})\cap \mathfrak{z}(\mathfrak{g})$ the \emph{period group} (which is independent of $\theta$).
\end{definition}
\begin{remark}\label{rem:periodGroupWithInvolution}
	Being $\theta$-invariant, the period group becomes a topological group with involution\footnote{To avoid confusion, we do not speak of a \emph{symmetric group}, since this term is usually used for the group of all bijections of a set.}, which we denote by $\Pi(\mathfrak{g},\theta):=(\Pi(\mathfrak{g}),\Pi(\theta))$ with $\Pi(\theta):=\theta|_{\Pi(\mathfrak{g})}^{\Pi(\mathfrak{g})}$. We denote the fixed point groups of the involutive automorphisms $\Pi(\theta)$ and $-\Pi(\theta)$ by $\Pi(\mathfrak{g})_+:=\Pi(\mathfrak{g})^{\Pi(\theta)}\subseteq \mathfrak{z}(\mathfrak{g})_+$ and $\Pi(\mathfrak{g})_-:=\Pi(\mathfrak{g})^{-\Pi(\theta)}\subseteq \mathfrak{z}(\mathfrak{g})_-$, respectively. We do not know if $\Pi(\mathfrak{g})$ necessarily has to be given by $\Pi(\mathfrak{g})_+ \oplus \Pi(\mathfrak{g})_-$,\footnote{Consider, for example, a Banach--Lie algebra $\mathfrak{h}$ with period group $\Pi(\mathfrak{h})=\ZZ$ (cf.\ e.g.\ \cite[Ex.~VI.3]{GN03}). Then the symmetric Lie algebra $(\mathfrak{g},\theta)$ defined by $\mathfrak{g}:=\mathfrak{h}\times \mathfrak{h}$ and $\theta(h_1,h_2):=(h_2,h_1)$ has period group $\Pi(\mathfrak{g})=\ZZ\times \ZZ$, but leads to $\Pi(\mathfrak{g})_+ + \Pi(\mathfrak{g})_-=(2\ZZ\times 2\ZZ) \cup ((1+2\ZZ)\times (1+2\ZZ))$.} but we know that
	\begin{equation} \label{eqn:Pi_+Pi_-}
		\Pi(\mathfrak{g})_+ \oplus \Pi(\mathfrak{g})_- \ \leq\ \Pi(\mathfrak{g})\ \leq\ \frac{1}{2}\Pi(\mathfrak{g})_+ \oplus \frac{1}{2}\Pi(\mathfrak{g})_-.
	\end{equation}
	Indeed, given any $g\in \Pi(\mathfrak{g})$, we have
	$$g \ =\ \frac{g+\theta(g)}{2} + \frac{g-\theta(g)}{2} \ \in\ \frac{1}{2}\Pi(\mathfrak{g})_+ \oplus \frac{1}{2}\Pi(\mathfrak{g})_-.$$
\end{remark}
\begin{remark} \label{rem:im(f)LieSubgroupIffPi(g)Discrete}
	The image $\im(f_{(\mathfrak{g},\theta)})$ of $f_{(\mathfrak{g},\theta)}$ is the connected normal integral subgroup of $\widehat\Omega(G)$ with Lie algebra $\Omega(\mathfrak{g}_{\ad})\hookrightarrow \Omega(\mathfrak{g}_{\ad})\times \mathfrak{z}(\mathfrak{g})=\widehat\Omega(\mathfrak{g})$. It is actually a Lie subgroup if and only if the period group $\Pi(\mathfrak{g})$ is discrete (cf.\ \cite{GN03} and also \cite[Prop.~2.6]{Nee02Central}). This is equivalent to the discreteness of $\Pi(\mathfrak{g})_+$ and $\Pi(\mathfrak{g})_-$ (cf.\ (\ref{eqn:Pi_+Pi_-})).
\end{remark}
\begin{remark} \label{rem:periodGroup=kernel}
	Given a 1-connected Lie group $G$ with Lie algebra $\mathfrak{g}$, the period group is given by  $\Pi(\mathfrak{g})=\ker(\exp_G|_{\mathfrak{z}(\mathfrak{g})})$ and is isomorphic to the fundamental group $\pi_1(Z(G))$ of the center of $G$ (cf.\ \cite[Prop.~III.8]{GN03}).
\end{remark}
%
%
%
%
\subsection{Integrability Criterion for Lie Triple Systems}
\label{sec:IntegrabilityCriterion}
In this subsection, we investigate the problem under which conditions a Lie triple system is \emph{integrable}, i.e., arises as the Lie triple system of a pointed symmetric space. Note that an integrable Lie triple system can always be integrated to a 1-connected symmetric space (cf.\ Remark~\ref{rem:Lts(covering)=Iso}), which is unique up to an isomorphism by the Integrability Theorem. First, we show that a Lie triple system is integrable if and only if its standard embedding is integrable. In a second step, we translate the classical result that a Lie algebra $\mathfrak{g}$ is integrable if and only if its period group $\Pi(\mathfrak{g})$ is discrete into the language of symmetric spaces.
%
\begin{lemma}\label{lem:BanachSpaceEmbeddingOfDer(M)_+}
	Let $(M,b)$ be a pointed connected symmetric space. Then the map
	$$\Der(M)_+ \rightarrow \gl(T_bM),\ \xi\mapsto (v\mapsto\nabla\!_v\xi)$$
	is an embedding of Banach spaces.
\end{lemma}
\begin{proof}
	Considering the embedding
	$$\Der(M)\rightarrow T_bM \times \gl(T_bM),\ \xi\mapsto \big(\xi(b),\ v\mapsto\nabla\!_v\xi\big)$$
	(cf.\ Section~\ref{sec:AutOfConnectedSymSpace}), the image of the closed subspace $\Der(M)_+$ lies in $\{0\}\times\gl(T_bM)$, so that the assertion follows by restricting this map.
\end{proof}
\begin{proposition}\label{prop:LieAlgebraEmbeddingOfDer(M)_+}
	Let $(M,b)$ be a pointed connected symmetric space. Then the map
	$$\Der(M)_+ \rightarrow \gl(\Der(M)_-),\ \xi\mapsto [\xi,\cdot]|_{\Der(M)_-}$$
	is an embedding of Banach--Lie algebras.
\end{proposition}
\begin{proof}
	Cf.\ \cite[p.~91]{Loo69} for the finite-dimensional case. Since the map is a representation of a Banach--Lie algebra (cf.\ Section~\ref{sec:symLieAlgAndLieGrp}), it remains to show that it is a closed embedding of Banach spaces.
	Considering the topological linear isomorphism $\ev_b\colon \Der(M)_-\rightarrow T_bM$, we have a
	topological linear isomorphism
	$$(\ev_b)_\ast\colon \gl(\Der(M)_-)\rightarrow \gl(T_bM),\ A\mapsto \big(v\mapsto \ev_b(A(\xi_v))\big),$$
	where $\xi_v:=(\ev_b)^{-1}(v)$.
	In the light of Lemma~\ref{lem:BanachSpaceEmbeddingOfDer(M)_+}, it suffices to show that, for each\linebreak $\xi\in\Der(M)_+$, the isomorphism $-(\ev_b)_\ast$ maps $[\xi,\cdot]|_{\Der(M)_-}\in\gl(\Der(M)_-)$ to $\nabla\!_\cdot\xi\in\gl(T_bM)$.
	Given any $v\in T_bM$, we have
	$$-(\ev_b)_\ast([\xi,\cdot]|_{\Der(M)_-})(v) \ =\ -\ev_b([\xi,\xi_v]) \ =\ -\ev_b\big(\nabla\!_\xi\xi_v-\nabla\!_{\xi_v}\xi-\Tor(\xi,\xi_v)\big),$$
	where $\Tor$ denotes the torsion tensor (cf.\ \cite{Klo09b}).
	Since the affine connection of a symmetric space is torsionfree and since $\xi\in\Der(M)_+$ entails $\xi(b)=0$ and hence $\nabla\!_{\xi(b)}\xi_v=0$, this expression simplifies to $\nabla\!_v\xi$.
\end{proof}
\begin{proposition} \label{prop:g'(M,b)IsomorphicS(m)}
	Let $(M,b)$ be a pointed connected symmetric space and abbreviate\linebreak $\mathfrak{m}:=\Der(M)_-$ (motivated by the isomorphism $\ev_b\colon\Der(M)_-\rightarrow \mathfrak{m}$). The symmetric Lie algebra
	$$\mathfrak{g}^\prime(M,b)=\overline{[\mathfrak{m},\mathfrak{m}]}\oplus\mathfrak{m}\ \leq\ \Der(M)_+\oplus\Der(M)_-$$
	(cf.\ Section~\ref{sec:AutOfConnectedSymSpace}) is isomorphic to the standard embedding $S(\mathfrak{m})$ of $\mathfrak{m}$ via the isomorphism
	$$\Phi\colon\overline{[\mathfrak{m},\mathfrak{m}]}\oplus\mathfrak{m} \rightarrow S(\mathfrak{m})_+\oplus S(\mathfrak{m})_-,\quad \xi_+\oplus\xi_-\mapsto [\xi_+,\cdot]|_{\mathfrak{m}}\oplus \xi_-.$$
\end{proposition}
\begin{proof}
	Cf.\ \cite[p.~91]{Loo69} for the finite-dimensional case.
	The closed embedding
	$$\Der(M)_+ \rightarrow \gl(\mathfrak{m}),\ \xi\mapsto [\xi,\cdot]|_{\mathfrak{m}}$$
	(cf.\ Proposition~\ref{prop:LieAlgebraEmbeddingOfDer(M)_+}) maps the subalgebra $[\mathfrak{m},\mathfrak{m}]$ onto the subalgebra
	$$\spann\{[x,y,\cdot]\in\gl(\mathfrak{m})\colon x,y\in\mathfrak{m}\}\leq \gl(\mathfrak{m})$$
	and hence the closure $\overline{[\mathfrak{m},\mathfrak{m}]}$ onto the closure, which is $S(\mathfrak{m})_+$. 
	Therefore, the partial map $\overline{[\mathfrak{m},\mathfrak{m}]}\rightarrow S(\mathfrak{m})_+$ of $\Phi$ is correctly defined and is an isomorphism of Banach--Lie algebras.
	Thus it is clear that $\Phi$ is a topological linear isomorphism. To see that $\Phi$ is a homomorphism of Lie algebras, it remains to consider the restrictions of the Lie bracket on $\mathfrak{g}^\prime(M,b)$ to the sets $\overline{[\mathfrak{m},\mathfrak{m}]}\times \mathfrak{m}$, $\mathfrak{m}\times \overline{[\mathfrak{m},\mathfrak{m}]}$ and $\mathfrak{m}\times \mathfrak{m}$, but then the proof is quite easy.
\end{proof}
\begin{corollary} \label{cor:S(Lts(M,b))IsIntegrable}
	Let $(M,b)$ be a pointed symmetric space with Lie triple system $\mathfrak{m}$. Then the standard embedding $S(\mathfrak{m})$ of $\mathfrak{m}$ is integrable.
\end{corollary}
\begin{proof}
	Since the symmetric Lie algebra $\mathfrak{g}^\prime(M,b)$ is integrable (cf.\ Section~\ref{sec:AutOfConnectedSymSpace} and Remark~\ref{rem:integrabilityOfSymLieAlgebra}), the standard embedding $S(\mathfrak{m})$ is integrable as well.
\end{proof}
\begin{lemma} \label{lem:integrableSymLieAlgebra}
	Let $(\mathfrak{g},\theta)$ be an integrable symmetric Lie algebra. Then also the Lie triple system $\mathfrak{g}_-$ is integrable.
\end{lemma}
\begin{proof}
	Let $(G,\sigma)$ be a symmetric Lie group with Lie algebra $(\mathfrak{g},\theta)$, then $G/G^\sigma$ is a symmetric space with Lie triple system $\mathfrak{g}_-$.
\end{proof}
\begin{theorem} \label{th:integrability m<=>S(m)}
	A Lie triple system $\mathfrak{m}$ is integrable to a pointed symmetric space if and only if its standard embedding $S(\mathfrak{m})$ is integrable to a (symmetric) Lie group.
\end{theorem}
\begin{proof}
	The theorem follows immediately by Corollary~\ref{cor:S(Lts(M,b))IsIntegrable} and Lemma~\ref{lem:integrableSymLieAlgebra}.
\end{proof}
\begin{corollary} \label{cor:finite-dimLtsIsIntegrable}
	Every finite-dimensional Lie triple system is integrable.
\end{corollary}
\begin{remark} \label{rem:m_ad}
	Let $\mathfrak{m}$ be a Lie triple system $\mathfrak{m}$ with center $\mathfrak{z}(\mathfrak{m})$. Then the quotient Lie triple system $\mathfrak{m}_{\ad}:=\mathfrak{m}/\mathfrak{z}(\mathfrak{m})$ is integrable.
	Indeed, since the center $\mathfrak{z}(S(\mathfrak{m}))$ of the standard embedding coincides with $\mathfrak{z}(\mathfrak{m})$ (cf.\ Section~\ref{sec:LtsAndSymSpace}), the Lie algebra $S(\mathfrak{m})_{\ad}:=S(\mathfrak{m})/\mathfrak{z}(S(\mathfrak{m}))$ is given by $S(\mathfrak{m})_+\oplus \mathfrak{m}_{\ad}$. The Lie algebra $S(\mathfrak{m})_{\ad}$ being integrable (cf.\ \cite{GN03}), the Lie triple system $\mathfrak{m}_{\ad}$ is integrable, too, by Lemma~\ref{lem:integrableSymLieAlgebra}.
\end{remark}
In the following, we refer to the way of defining the period morphism $\per_{(\mathfrak{g},\theta)}$ and the period group $\Pi(\mathfrak{g}):=\im(\per_{(\mathfrak{g},\theta)})$ of a symmetric Lie algebra $(\mathfrak{g},\theta)$ as presented in Section~\ref{sec:periodMorphismOfSymLieAlgebra}. We aim to apply those considerations to the standard embedding $S(\mathfrak{m})$ of a Lie triple\linebreak system $\mathfrak{m}$ (or more generally to any symmetric Lie algebra $(\mathfrak{g},\theta)$ with $\mathfrak{g}_-=\mathfrak{m}$ and $\mathfrak{z}(\mathfrak{g})=\mathfrak{z}(\mathfrak{g}_-)$) in order to deduce similar terms and results for Lie triple systems. By the way, we obtain a further proof (and moreover a generalization) of Theorem~\ref{th:integrability m<=>S(m)}. 

Before we start with a Lie triple system, we firstly consider any symmetric Lie algebra $(\mathfrak{g},\theta)$. The following proposition provides the decisive link to translate the knowledge of Section~\ref{sec:periodMorphismOfSymLieAlgebra} into the setting of Lie triple systems.
\begin{proposition} \label{prop:diagramOfSymPairs}
	Let $(\mathfrak{g},\theta)$ be a symmetric Lie algebra.
	Then the exact sequences of morphisms of symmetric Lie groups defined in Section~\ref{sec:periodMorphismOfSymLieAlgebra} can be turned into exact sequences of morphisms of symmetric pairs that lead to the following commutative diagram:
	$$\begin{xy}
		\xymatrixcolsep{1pc} 
		\xymatrix{
			& 0\ar[d] & 0\ar[d]\\
			&
			(\mathfrak{z}(\mathfrak{g},\theta),\mathfrak{z}(\mathfrak{g})_+)\ar@{=}[r]\ar@{^{(}->}[d]
			&
			(\mathfrak{z}(\mathfrak{g},\theta),\mathfrak{z}(\mathfrak{g})_+)\ar@{^{(}->}[d] \\
			\eins\ar[r] & (\widehat\Omega(G,\sigma),(\widehat P(G)^{\widehat P(\sigma)})_0 \cap \widehat\Omega(G))\ar@{^{(}->}[r]\ar[d]
			&
			(\widehat P(G,\sigma), (\widehat P(G)^{\widehat P(\sigma)})_0)\ar[r]\ar[d]^{\chi_{(\mathfrak{g},\theta)}}
			&
			(G_{\ad},\sigma_{\ad}, (G_{\ad}^{\sigma_{\ad}})_0)\ar[r]\ar@{=}[d]
			& \eins \\
			\eins\ar[r] & (\Omega(G_{\ad},\sigma_{\ad}),\Omega(G_{\ad}^{\sigma_{\ad}}))\ar@{^{(}->}[r]\ar[d]
			&
			(P(G_{\ad},\sigma_{\ad}),P(G_{\ad}^{\sigma_{\ad}}))\ar[r]^-{\ev_1}\ar[d]
			&
			(G_{\ad},\sigma_{\ad},(G_{\ad}^{\sigma_{\ad}})_0)\ar[r]
			&
			\eins \\
			& \eins & \eins
		}
	\end{xy}$$
\end{proposition}
\begin{proof}
	Because of $\widehat\Omega(G)^{\widehat\Omega(\sigma)}= \widehat P(G)^{\widehat P(\sigma)} \cap \widehat\Omega(G)$, we have
	\begin{equation*}
		(\widehat\Omega(G)^{\widehat\Omega(\sigma)})_0 \ \subseteq\ (\widehat P(G)^{\widehat P(\sigma)})_0 \cap \widehat\Omega(G) \ \subseteq\ \widehat\Omega(G)^{\widehat\Omega(\sigma)},
	\end{equation*}
	so that $(\widehat\Omega(G,\sigma),(\widehat P(G)^{\widehat P(\sigma)})_0 \cap \widehat\Omega(G))$
	is indeed a symmetric pair. Thus, the above diagram consists of morphisms of symmetric pairs and we need only prove its exactness.
	
	To see the exactness of the sequence
	\begin{equation} \label{eqn:sequenceHatP(G,sigma)WithPairs}
		0\rightarrow (\mathfrak{z}(\mathfrak{g},\theta),\mathfrak{z}(\mathfrak{g})_+) \hookrightarrow (\widehat P(G,\sigma),(\widehat P(G)^{\widehat P(\sigma)})_0)\rightarrow (P(G_{\ad},\sigma_{\ad}),P(G_{\ad}^{\sigma_{\ad}}))\rightarrow \eins
	\end{equation}
	of symmetric pairs, we must check that we have $\mathfrak{z}(\mathfrak{g})_+ = (\widehat P(G)^{\widehat P(\sigma)})_0 \cap \mathfrak{z}(\mathfrak{g})$ and that the map $(\widehat P(G)^{\widehat P(\sigma)})_0\rightarrow P(G_{\ad}^{\sigma_{\ad}})$ is surjective (cf.\ Remark~\ref{rem:characterizeExactness}). The former is clear by Proposition~\ref{prop:symToSurjectiveInjevtive}(2) and the connectedness of $\mathfrak{z}(\mathfrak{g})_+$ admits only one symmetric pair for $\mathfrak{z}(\mathfrak{g},\theta)$.
	To see the latter, it suffices to check that the morphism $\widehat P(\mathfrak{g},\theta)\rightarrow P(\mathfrak{g}_{\ad},\theta_{\ad})$ maps $\widehat P(\mathfrak{g})_+$ onto $P(\mathfrak{g}_{\ad})_+=P((\mathfrak{g}_{\ad})_+)$ (cf.\ p.~\pageref{page:L(f)(g1_+)=g2_+}), but this is clear, since we have $(\mathfrak{g}_{\ad})_+=\mathfrak{g}_+ / \mathfrak{z}(\mathfrak{g})_+$ and hence $\widehat P(\mathfrak{g})_+ = \{(\gamma,x)\in P((\mathfrak{g}_{\ad})_+)\times\mathfrak{g}_+\colon \gamma(1)=x+\mathfrak{z}(\mathfrak{g})_+\}$.
			
	To see the exactness of the sequence
	\begin{equation*}
		\eins \rightarrow (\widehat\Omega(G,\sigma),(\widehat P(G)^{\widehat P(\sigma)})_0 \cap \widehat\Omega(G))\hookrightarrow (\widehat P(G,\sigma), (\widehat P(G)^{\widehat P(\sigma)})_0) \rightarrow (G_{\ad},\sigma_{\ad}, (G_{\ad}^{\sigma_{\ad}})_0) \rightarrow \eins
	\end{equation*}
 	of symmetric pairs, we have to verify the exactness of the sequence
	\begin{equation*} 
		\eins \rightarrow (\widehat P(G)^{\widehat P(\sigma)})_0 \cap \widehat\Omega(G) \hookrightarrow (\widehat P(G)^{\widehat P(\sigma)})_0 \rightarrow (G_{\ad}^{\sigma_{\ad}})_0 \rightarrow \eins
	\end{equation*}
	of Lie groups.
	The homomorphism $(\widehat P(G)^{\widehat P(\sigma)})_0 \rightarrow (G_{\ad}^{\sigma_{\ad}})_0$ is surjective, since it is given by the composition of the surjective map $(\widehat P(G)^{\widehat P(\sigma)})_0 \rightarrow P(G_{\ad}^{\sigma_{\ad}})$ (cf.\ (\ref{eqn:sequenceHatP(G,sigma)WithPairs})) and the surjective evaluation map $P(G_{\ad}^{\sigma_{\ad}})\rightarrow (G_{\ad}^{\sigma_{\ad}})_0$. Its kernel is obviously given by $(\widehat P(G)^{\widehat P(\sigma)})_0 \cap \widehat\Omega(G)$, so that the sequence is exact.
	
	It is clear that the sequence
	$$\eins\rightarrow (\Omega(G_{\ad},\sigma_{\ad}),\Omega(G_{\ad}^{\sigma_{\ad}})) \hookrightarrow (P(G_{\ad},\sigma_{\ad}),P(G_{\ad}^{\sigma_{\ad}}))  \stackrel{\ev_1}{\rightarrow} (G_{\ad},\sigma_{\ad},(G_{\ad}^{\sigma_{\ad}})_0) \rightarrow \eins$$
	is exact, so that it remains to show the exactness of the sequence
	\begin{equation*}
	0\rightarrow (\mathfrak{z}(\mathfrak{g},\theta),\mathfrak{z}(\mathfrak{g})_+)\hookrightarrow (\widehat\Omega(G,\sigma),(\widehat P(G)^{\widehat P(\sigma)})_0 \cap \widehat\Omega(G))\rightarrow (\Omega(G_{\ad},\sigma_{\ad}),\Omega(G_{\ad}^{\sigma_{\ad}}))\rightarrow \eins
	\end{equation*}
	of symmetric pairs. Because of the connectedness of $\mathfrak{z}(\mathfrak{g})_+$, the same argument as above shows that, for this, we need only check that the map $(\widehat P(G)^{\widehat P(\sigma)})_0 \cap \widehat\Omega(G)\rightarrow \Omega(G_{\ad}^{\sigma_{\ad}})$ is surjective, but this is true, since this map arises in the commutative and exact diagram
	$$\begin{xy}
		\xymatrix{
			\eins\ar[r] & (\widehat P(G)^{\widehat P(\sigma)})_0 \cap \widehat\Omega(G)\ar@{^{(}->}[r]\ar[d]
			&
			(\widehat P(G)^{\widehat P(\sigma)})_0\ar[r]\ar[d]^{\text{surjective}}
			&
			(G_{\ad}^{\sigma_{\ad}})_0\ar[r]\ar@{=}[d]
			& \eins \\
			\eins\ar[r] & \Omega(G_{\ad}^{\sigma_{\ad}})\ar@{^{(}->}[r]
			&
			P(G_{\ad}^{\sigma_{\ad}})\ar[r]^-{\ev_1}
			&
			(G_{\ad}^{\sigma_{\ad}})_0\ar[r]
			&
			\eins\lefteqn{.} 
		}
	\end{xy}$$
\end{proof}
\begin{corollary}\label{cor:diagramOfQuotients}
	Applying the exact functor $\Sym$ (cf.\ Propostion~\ref{prop:exactnessOfSym}) to the diagram in Proposition~\ref{prop:diagramOfSymPairs} leads to the commutative and exact diagram
	$$\begin{xy}
		\xymatrixcolsep{1pc} 
		\xymatrix{
			& o\ar[d] & o\ar[d]\\
			&
			\mathfrak{z}(\mathfrak{g})/\mathfrak{z}(\mathfrak{g})_+\ar@{=}[r]\ar@{^{(}->}[d]
			&
			\mathfrak{z}(\mathfrak{g})/\mathfrak{z}(\mathfrak{g})_+\ar@{^{(}->}[d] \\
			o\ar[r] & \widehat\Omega(G)/((\widehat P(G)^{\widehat P(\sigma)})_0 \cap \widehat\Omega(G))\ar@{^{(}->}[r]\ar[d]
			&
			\widehat P(G)/(\widehat P(G)^{\widehat P(\sigma)})_0\ar[r]\ar[d]
			&
			G_{\ad}/(G_{\ad}^{\sigma_{\ad}})_0\ar[r]\ar@{=}[d]
			& o \\
			o\ar[r] & \Omega(G_{\ad})/\Omega(G_{\ad}^{\sigma_{\ad}})\ar@{^{(}->}[r]\ar[d]
			&
			P(G_{\ad})/P(G_{\ad}^{\sigma_{\ad}})\ar[r]\ar[d]
			&
			G_{\ad}/(G_{\ad}^{\sigma_{\ad}})_0\ar[r]
			&
			o \\
			& o & o
		}
	\end{xy}$$
	of pointed symmetric spaces, where all arising inclusion maps are topological embeddings.
\end{corollary}
\begin{remark} \label{rem:z(g)/z(g)_+}
	The quotient $\mathfrak{z}(\mathfrak{g})/\mathfrak{z}(\mathfrak{g})_+$ can be identified with the pointed symmetric space $\mathfrak{z}(\mathfrak{g})_-$. Indeed, the multiplication map is given by $x\cdot y := x-\theta(x)+\theta(y) = 2x - y$.
\end{remark}
%
\begin{corollary} \label{cor:diagramOfLtsOfQuotients}
	Applying the Lie functor $\Lts$ to the diagram in Corollary~\ref{cor:diagramOfQuotients} leads to the commutative and exact diagram
	$$\begin{xy}
	\xymatrix{
		& 0\ar[d] & 0\ar[d]\\
		&
		\mathfrak{z}(\mathfrak{g})_-\ar@{=}[r]\ar@{^{(}->}[d]
		&
		\mathfrak{z}(\mathfrak{g})_-\ar@{^{(}->}[d] \\
		0\ar[r] & \widehat\Omega(\mathfrak{g})_-\ar@{^{(}->}[r]\ar[d]
		&
		\widehat P(\mathfrak{g})_-\ar[r]\ar[d]
		&
		(\mathfrak{g}_{\ad})_-\ar[r]\ar@{=}[d]
		& 0 \\
		0\ar[r] & \Omega(\mathfrak{g}_{\ad})_-\ar@{^{(}->}[r]\ar[d]
		&
		P(\mathfrak{g}_{\ad})_-\ar[r]^-{\ev_1}\ar[d]
		&
		(\mathfrak{g}_{\ad})_-\ar[r]
		& 0 \\
		& 0 & 0
	}
	\end{xy}$$
	of Lie triple systems. Note that all inclusion maps are topological embeddings.
\end{corollary}
\begin{proof}
	This follows by $\Lts\circ\Sym = \calLts\circ L \circ F$ (cf.\ (\ref{eqn:LtsSym=calLtsL}).
\end{proof}
Now, let $\mathfrak{m}$ be a Lie triple system. Consider some symmetric Lie algebra $(\mathfrak{g},\theta)$ with $\mathfrak{g}_-=\mathfrak{m}$ and $\mathfrak{z}(\mathfrak{g})=\mathfrak{z}(\mathfrak{m})$, e.g.\ the standard embedding $S(\mathfrak{m})$ of $\mathfrak{m}$.
Note that $\mathfrak{z}(\mathfrak{g},\theta)$ then stands for the symmetric Lie algebra $(\mathfrak{z}(\mathfrak{g}),-\id_{\mathfrak{z}(\mathfrak{g})})$ with $\mathfrak{z}(\mathfrak{g})_+=\{0\}$ and $\mathfrak{z}(\mathfrak{g})_-=\mathfrak{z}(\mathfrak{m})$.

The central extension
$$0\rightarrow \mathfrak{z}(\mathfrak{m})\hookrightarrow \mathfrak{m}\rightarrow \mathfrak{m}_{\ad}\rightarrow 0$$
can be pulled back via the evaluation morphism $\ev_1\colon P(\mathfrak{m}_{\ad})\rightarrow \mathfrak{m}_{\ad}$ to a central extension
\begin{equation} \label{eqn:extensionHatP(m)}
	0\rightarrow \mathfrak{z}(\mathfrak{m})\hookrightarrow \widehat P(\mathfrak{m})\rightarrow P(\mathfrak{m}_{\ad})\rightarrow 0
\end{equation}
with $$\widehat P(\mathfrak{m}):=\{(\gamma,x)\in P(\mathfrak{m}_{\ad})\times\mathfrak{m}\colon\gamma(1)=x+\mathfrak{z}(\mathfrak{m})\}.$$
Restricting it to the preimage $\widehat\Omega(\mathfrak{m}):=\Omega(\mathfrak{m}_{\ad})\times\mathfrak{z}(\mathfrak{m})$ of $\Omega(\mathfrak{m}_{\ad})$, we obtain a central extension
\begin{equation} \label{eqn:extsenionHatOmega(m)}
	0\rightarrow \mathfrak{z}(\mathfrak{m})\hookrightarrow \widehat \Omega(\mathfrak{m})\rightarrow \Omega(\mathfrak{m}_{\ad})\rightarrow 0.
\end{equation}
%
%
\begin{remark}\label{rem:diagramOfLts}
	The diagram in Corollary~\ref{cor:diagramOfLtsOfQuotients} simplifies to
	$$\begin{xy}
		\xymatrix{
			& 0\ar[d] & 0\ar[d]\\
			&
			\mathfrak{z}(\mathfrak{m})\ar@{=}[r]\ar@{^{(}->}[d]
			&
			\mathfrak{z}(\mathfrak{m})\ar@{^{(}->}[d] \\
			0\ar[r] & \widehat\Omega(\mathfrak{m})\ar@{^{(}->}[r]\ar[d]
			&
			\widehat P(\mathfrak{m})\ar[r]\ar[d]
			&
			\mathfrak{m}_{\ad}\ar[r]\ar@{=}[d]
			& 0 \\
			0\ar[r] & \Omega(\mathfrak{m}_{\ad})\ar@{^{(}->}[r]\ar[d]
			&
			P(\mathfrak{m}_{\ad})\ar[r]^-{\ev_1}\ar[d]
			&
			\mathfrak{m}_{\ad}\ar[r]
			& 0 \\
			& 0 & 0
		}
	\end{xy}$$
	This follows easily by the decompositions $\mathfrak{g}= \mathfrak{g}_+\oplus\mathfrak{m}$, $\mathfrak{z}(\mathfrak{g})=0\oplus\mathfrak{z}(\mathfrak{m})$ and $\mathfrak{g}_{\ad}=\mathfrak{g}_+\oplus\mathfrak{m}_{\ad}$.
\end{remark}
We define the pointed symmetric space $(M_{\ad},b_{\ad}):=G_{\ad}/(G_{\ad}^{\sigma_{\ad}})_0$ and note that it is 1-connected (cf.\ Proposition~\ref{prop:1-connectedQuotient}) and that $\mathfrak{m}_{\ad}$ is its Lie triple system (cf.\ Remark~\ref{rem:diagramOfLts}). We shall frequently consider the Banach space $\mathfrak{z}(\mathfrak{m})$ endowed with its natural structure of a pointed symmetric space and note that $\Lts(\mathfrak{z}(\mathfrak{m}))=\mathfrak{z}(\mathfrak{m})$.
\begin{proposition} \label{prop:chi_m}
	Let $\mathfrak{m}$ be a Lie triple system and $(\mathfrak{g},\theta)$ a symmetric Lie algebra with $\mathfrak{g}_-=\mathfrak{m}$ and $\mathfrak{z}(\mathfrak{g})=\mathfrak{z}(\mathfrak{m})$.
	Then there exists an exact sequence
	$$o\rightarrow\mathfrak{z}(\mathfrak{m})\hookrightarrow \widehat P(M,b)\stackrel{\chi_{\mathfrak{m}}}{\rightarrow} P(M_{\ad},b_{\ad})\rightarrow o$$
	of pointed symmetric spaces corresponding to (\ref{eqn:extensionHatP(m)}), for a suitable pointed 1-connected symmetric space $\widehat P(M,b)$, where the inclusion of $\mathfrak{z}(\mathfrak{m})$ into $\widehat P(M,b)$ is a topological embedding.\footnote{Following the notation of Section~\ref{sec:periodMorphismOfSymLieAlgebra} and \cite{GN03}, we use notations like $\widehat P(M,b)$ and (presently) $\widehat \Omega (M,b)$ although there need not be a pointed symmetric space $(M,b)$ with Lie triple system $\mathfrak{m}$.}
\end{proposition}
\begin{proof}
	We put $\widehat P(M,b) := \widehat P(G)/(\widehat P(G)^{\widehat P(\sigma)})_0$, which is 1-connected by Proposition~\ref{prop:1-connectedQuotient}. The exact sequence
	\begin{equation*} 
		o\rightarrow\mathfrak{z}(\mathfrak{g})/\{0\} \hookrightarrow \widehat P(M,b)\rightarrow P(G_{\ad})/P(G_{\ad}^{\sigma_{\ad}})\rightarrow o
	\end{equation*}
	given in Corollary~\ref{cor:diagramOfQuotients} corresponds to (\ref{eqn:extensionHatP(m)}) by Remark~\ref{rem:diagramOfLts}.
	By Theorem~\ref{th:pathSymSpace}(2), we can replace the quotient $P(G_{\ad})/P(G_{\ad}^{\sigma_{\ad}})$ by $P(M_{\ad},b_{\ad})$, since we have defined $(M_{\ad},b_{\ad})$ as the quotient $G_{\ad}/(G_{\ad}^{\sigma_{\ad}})_0$. By Remark~\ref{rem:z(g)/z(g)_+}, we have $\mathfrak{z}(\mathfrak{g})/\{0\}\cong\mathfrak{z}(\mathfrak{m})$, entailing the assertion. Note that the assertion that the inclusion is a topological embedding also follows automatically in the light of Proposition~\ref{prop:integralSubreflectionSpace} and Proposition~\ref{prop:kernelOfMorphismOfPointedSymSpaces}.
\end{proof}
Let us consider the morphism
$$\ev_1^{(M_{\ad},b_{\ad})}\circ\chi_{\mathfrak{m}}\colon \widehat P(M,b)\rightarrow (M_{\ad},b_{\ad}),\ x\mapsto \chi_{\mathfrak{m}}(x)(1)$$
of pointed symmetric spaces. Its kernel is a symmetric subspace $\widehat \Omega(M,b)\unlhd \widehat P(M,b)$ whose Lie triple system is the kernel of the morphism
$$\ev_1^{\mathfrak{m}_{\ad}}\circ \Lts(\chi_{\mathfrak{m}})\colon \widehat P(\mathfrak{m})\rightarrow \mathfrak{m}_{\ad},\ (\gamma,x) \mapsto \gamma(1)=x+\mathfrak{z}(\mathfrak{m})$$
of Lie triple systems (cf.\ Proposition~\ref{prop:kernelOfMorphismOfPointedSymSpaces}), i.e., is the Lie triple system $\widehat\Omega(\mathfrak{m})=\Omega(\mathfrak{m}_{\ad})\times\mathfrak{z}(\mathfrak{m})$. Since we have $\widehat\Omega(M,b)=(\chi_\mathfrak{m})^{-1}(\Omega(M_{\ad},b_{\ad}))$, 
the exact sequence of Proposition~\ref{prop:chi_m} can be restricted to the exact sequence
\begin{equation} \label{eqn:restrictionOfChi^m}
	o\rightarrow\mathfrak{z}(\mathfrak{m})\hookrightarrow \widehat \Omega(M,b)\rightarrow \Omega(M_{\ad},b_{\ad})\rightarrow o,
\end{equation}
which corresponds to (\ref{eqn:extsenionHatOmega(m)}).
\begin{remark} \label{rem:diagramSymSpace}
	The commutative and exact diagram
	$$\begin{xy}
		\xymatrix{
			& o\ar[d] & o\ar[d]\\
			&
			\mathfrak{z}(\mathfrak{m})\ar@{=}[r]\ar@{^{(}->}[d]
			&
			\mathfrak{z}(\mathfrak{m})\ar@{^{(}->}[d] \\
			o\ar[r] & \widehat\Omega(M,b)\ar@{^{(}->}[r]\ar[d]
			&
			\widehat P(M,b)\ar[r]\ar[d]^{\chi_{\mathfrak{m}}}
			&
			(M_{\ad},b_{\ad})\ar[r]\ar@{=}[d]
			& o \\
			o\ar[r] & \Omega(M_{\ad},b_{\ad})\ar@{^{(}->}[r]\ar[d]
			&
			P(M_{\ad},b_{\ad})\ar[r]^-{\ev_1}\ar[d]
			&
			(M_{\ad},b_{\ad})\ar[r]
			&
			o \\
			& o & o
		}
	\end{xy}$$
	of pointed symmetric spaces gives a good overview. It can be identified with the diagram given in Corollary~\ref{cor:diagramOfQuotients} (with $\mathfrak{z}(\mathfrak{g})_+=\{0\}$) and corresponds to the diagram given in Remark~\ref{rem:diagramOfLts}.
\end{remark}
\begin{remark}
	Being isomorphic to $\widehat\Omega(G)/((\widehat P(G)^{\widehat P(\sigma)})_0 \cap \widehat\Omega(G))$, the symmetric space\linebreak $\widehat\Omega(M,b)$ is connected, since $\widehat\Omega(G)$ is connected.
\end{remark}
The Lie triple system $\Omega(\mathfrak{m}_{\ad})$ being integrable (to $\Omega(M_{\ad},b_{\ad})$), there is a pointed\linebreak 1-connected symmetric space $\widetilde\Omega(M_{\ad},b_{\ad})$ whose Lie triple system is $\Omega(\mathfrak{m}_{\ad})$. We integrate the inclusion morphism
$$\Omega(\mathfrak{m}_{\ad})\hookrightarrow \Omega(\mathfrak{m}_{\ad})\times\mathfrak{z}(\mathfrak{m})=\widehat\Omega(\mathfrak{m}),\ x\mapsto (x,0) $$
to a morphism
$$f_\mathfrak{m}\colon \widetilde\Omega(M_{\ad},b_{\ad})\rightarrow \widehat\Omega(M,b).$$
Composing it with $\widehat\Omega(M,b)\rightarrow \Omega(M_{\ad},b_{\ad})$ (cf.\ (\ref{eqn:restrictionOfChi^m})) gives us a morphism
$$q_{\Omega(M_{\ad},b_{\ad})}\colon \widetilde\Omega(M_{\ad},b_{\ad})\stackrel{f_\mathfrak{m}}{\rightarrow} \widehat\Omega(M,b) \rightarrow \Omega(M_{\ad},b_{\ad})$$
satisfying $\Lts(q_{\Omega(M_{\ad},b_{\ad})})=\id_{\Omega(\mathfrak{m}_{\ad})}$, so that it is a universal covering morphism (cf.\ Remark~\ref{rem:Lts(covering)=Iso}). Since $\mathfrak{z}(\mathfrak{m})$ is the kernel of $\widehat\Omega(M,b)\rightarrow \Omega(M_{\ad},b_{\ad})$, the kernel  of $q_{\Omega(M_{\ad},b_{\ad})}$ is mapped by $f_\mathfrak{m}$ onto $\im(f_\mathfrak{m})\cap \mathfrak{z}(\mathfrak{m})$. 
\begin{definition}\label{def:per_m}
	The restriction $\per_{\mathfrak{m}}\colon \ker(q_{\Omega(M_{\ad},b_{\ad})})\rightarrow \mathfrak{z}(\mathfrak{m})$ of $f_\mathfrak{m}$ is called the \emph{period morphism of $\mathfrak{m}$}. We denote its image by $\Pi(\mathfrak{m}):=\im(\per_{\mathfrak{m}}) = \im(f_\mathfrak{m})\cap \mathfrak{z}(\mathfrak{m})$.
\end{definition}
\begin{lemma} \label{lem:Sym(f_g)=f_m}
	If we turn the morphism
	$f_{(\mathfrak{g},\theta)}\colon \widetilde\Omega(G_{\ad},\sigma_{\ad})\rightarrow \widehat\Omega(G,\sigma)$
	of symmetric Lie groups (cf.\ (\ref{eqn:f_(g,theta)})) into the morphism
	$$f_{(\mathfrak{g},\theta)}\colon (\widetilde\Omega(G_{\ad},\sigma_{\ad}),(\widetilde\Omega(G_{\ad})^{\widetilde\Omega(\sigma_{\ad})})_0)\rightarrow (\widehat\Omega(G,\sigma),(\widehat P(G)^{\widehat P(\sigma)})_0 \cap \widehat\Omega(G))$$
	of symmetric pairs, then $\Sym(f_{(\mathfrak{g},\theta)})$ can be identified with $f_\mathfrak{m}$.
\end{lemma}
\begin{proof}
	Applying the functor $\Sym$ leads to a morphism
	$$\Sym(f_{(\mathfrak{g},\theta)})\colon \widetilde\Omega(G_{\ad})/(\widetilde\Omega(G_{\ad})^{\widetilde\Omega(\sigma_{\ad})})_0\rightarrow \widehat\Omega(G)/((\widehat P(G)^{\widehat P(\sigma)})_0 \cap \widehat\Omega(G))$$
	of pointed symmetric spaces.
	We identify $\widehat\Omega(G)/((\widehat P(G)^{\widehat P(\sigma)})_0 \cap \widehat\Omega(G))$ with $\widehat\Omega(M,b)$ (cf.\ Remark~\ref{rem:diagramSymSpace}) and we can identify the 1-connected space $\widetilde\Omega(G_{\ad})/(\widetilde\Omega(G_{\ad})^{\widetilde\Omega(\sigma_{\ad})})_0$ (cf.\ Proposition~\ref{prop:1-connectedQuotient}) with $\widetilde\Omega(M_{\ad},b_{\ad})$, since its Lie triple system is $\Omega(\mathfrak{g}_{\ad})_-=\Omega(\mathfrak{m}_{\ad})$.
	
	Because $\Lts(\Sym(f_{(\mathfrak{g},\theta)}))=\calLts(L(f_{(\mathfrak{g},\theta)}))$ is given by the inclusion morphism
	$$\Omega(\mathfrak{g}_{\ad})_-\hookrightarrow \Omega(\mathfrak{g}_{\ad})_-\times\mathfrak{z}(\mathfrak{g})_-=\widehat\Omega(\mathfrak{g})_-,\ x\mapsto (x,0),$$
	it coincides with $\Lts(f_\mathfrak{m})$, so that the assertion follows by uniqueness of integration.
\end{proof}
\begin{remark} \label{rem:diagram f_g f_m}
	We obtain the following commutative diagram
	$$\begin{xy}
		\xymatrix{
			&&&	\mathfrak{z}(\mathfrak{m})\ar@{^{(}->}[dd] \\
			&& (\mathfrak{z}(\mathfrak{g}),-\id_{\mathfrak{z}(\mathfrak{g})})\ar@{^{(}->}[dd]\ar[ru]^{q_1} \\
			&\widetilde\Omega(M_{\ad},b_{\ad})\ar[rr]^<<<<<<<<<{f_\mathfrak{m}}
			&& \widehat\Omega(M,b) \\
			\widetilde\Omega(G_{\ad},\sigma_{\ad})\ar[rr]^-{f_{(\mathfrak{g},\theta)}}\ar[ru]^{q_2}
			&& \widehat\Omega(G,\sigma)\ar[ru]^{q_3}
		}
	\end{xy}$$
	of smooth maps, where the maps $q_i$ with $i=1,2,3$ are the quotient maps given by the corresponding symmetric pairs (cf.\ Remark~\ref{rem:diagramSymSpace} and Lemma~\ref{lem:Sym(f_g)=f_m}). Note that $q_1$ is given by the symmetric pair $(\mathfrak{z}(\mathfrak{g}),-\id_{\mathfrak{z}(\mathfrak{g})},\{0\})$ and hence maps $\mathfrak{z}(\mathfrak{g})$ identically onto $\mathfrak{z}(\mathfrak{m})$ (cf.\ also Remark~\ref{rem:z(g)/z(g)_+}).
\end{remark}
\begin{proposition} \label{prop:Pi(g)inPi(m)}
	Let $\mathfrak{m}$ be a Banach--Lie triple system. Considering some symmetric Banach--Lie algebra $(\mathfrak{g},\theta)$ with $\mathfrak{g}_-=\mathfrak{m}$ and $\mathfrak{z}(\mathfrak{g})=\mathfrak{z}(\mathfrak{m})$ (e.g.\ the standard embedding $S(\mathfrak{m})$ of $\mathfrak{m}$), the image $\Pi(\mathfrak{m})$ of the period morphism $\per_\mathfrak{m}$ is given by
	$$\Pi(\mathfrak{m}) \ =\ \big(\im(f_{(\mathfrak{g},\theta)})((\widehat P(G)^{\widehat P(\sigma)})_0\cap\widehat\Omega(G))\big) \cap \mathfrak{z}(\mathfrak{g}).$$
	In particular, we have $\Pi(\mathfrak{g})\subseteq \Pi(\mathfrak{m})$.
\end{proposition}
\begin{proof}
	With the aid of the diagram in Remark~\ref{rem:diagram f_g f_m}, we obtain
	$$\Pi(\mathfrak{m}) \ =\ \im(f_\mathfrak{m})\cap\mathfrak{z}(\mathfrak{m}) \ =\ q_3(\im(f_{(\mathfrak{g},\theta)}))\cap\mathfrak{z}(\mathfrak{m}).$$
	Identifying the map $\mathfrak{z}(\mathfrak{m})\hookrightarrow \widehat\Omega(M,b)$ with $\mathfrak{z}(\mathfrak{g})/\{0\}\hookrightarrow \widehat\Omega(G)/((\widehat P(G)^{\widehat P(\sigma)})_0\cap\widehat\Omega(G))$ (cf.\ Remark~\ref{rem:diagramSymSpace}) leads to
	\[\Pi(\mathfrak{m}) \ =\ \big(\im(f_{(\mathfrak{g},\theta)})((\widehat P(G)^{\widehat P(\sigma)})_0\cap\widehat\Omega(G))\big) \cap\mathfrak{z}(\mathfrak{g}) \ \supseteq\ \im(f_{(\mathfrak{g},\theta)})\cap\mathfrak{z}(\mathfrak{g}) \ =\ \Pi(\mathfrak{g}).\qedhere\]
\end{proof}
\begin{corollary}
	The subset $\Pi(\mathfrak{m})\subseteq \mathfrak{z}(\mathfrak{m})$ is a subgroup.
\end{corollary}
\begin{proof}
	The product $\im(f_{(\mathfrak{g},\theta)})((\widehat P(G)^{\widehat P(\sigma)})_0\cap\widehat\Omega(G))$ of subgroups of $\widehat\Omega(G)$ is itself a subgroup, since $\im(f_{(\mathfrak{g},\theta)})\leq\widehat\Omega(G)$ is normal (cf.\ Remark~\ref{rem:im(f)LieSubgroupIffPi(g)Discrete}).
	Thus, its intersection with $\mathfrak{z}(\mathfrak{g})$ is a subgroup of $\mathfrak{z}(\mathfrak{g})$.
\end{proof}
\begin{definition}
	The image $\Pi(\mathfrak{m})$ of the period morphism $\per_\mathfrak{m}$ is called the \emph{period group of $\mathfrak{m}$}.
\end{definition}
\begin{remark}
	Recall that, in Definition~\ref{def:per_m}, the subset $\Pi(\mathfrak{m})\subseteq \mathfrak{z}(\mathfrak{m})$ was defined independently of the choice of $(\mathfrak{g},\theta)$. Hence, also its group structure is independent of $(\mathfrak{g},\theta)$, although the formula in Proposition~\ref{prop:Pi(g)inPi(m)} depends on $(\mathfrak{g},\theta)$.
\end{remark}
\begin{remark}
	Although the period morphism of $\mathfrak{m}$ is defined in a similar way as the period morphism $\per_{(\mathfrak{g},\theta)}\colon \pi_1(\Omega(G_{\ad},\sigma_{\ad}))\rightarrow\mathfrak{z}(\mathfrak{g},\theta)$ in Section~\ref{sec:periodMorphismOfSymLieAlgebra}, the following subtlety arises: It is not possible to turn the period morphism $\per_{(\mathfrak{g},\theta)}$ into a morphism of symmetric pairs that, applying the functor $\Sym$, leads to the period morphism $\per_{\mathfrak{m}}$. The contrary would be nice, because then the period groups $\Pi(\mathfrak{g})$ and $\Pi(\mathfrak{m})$ would coincide. However, we shall see in Example~\ref{ex:u(H)} that the case $\Pi(\mathfrak{g})\subsetneq \Pi(\mathfrak{m})$ can occur.
\end{remark}
\pagebreak
\begin{proposition}[Functoriality of the period group] \label{prop:functorialityOfPeriodSpace}
	Let $A\colon \mathfrak{m}\rightarrow \bar{\mathfrak{m}}$ be a morphism of Lie triple systems with $A(\mathfrak{z}(\mathfrak{m}))\subseteq\mathfrak{z}(\bar{\mathfrak{m}})$. 
	In view of the diagrams given in Remark~\ref{rem:diagramSymSpace}, the morphism $A$ induces a commutative diagram as follows:
	$$\begin{xy}
		\xymatrixcolsep{-1pc} 
		\xymatrix{
			&\ker(q_{\Omega(\bar M_{\ad},\bar b_{\ad})})\ar@{-->}[rr]^-{\per_{\bar{\mathfrak{m}}}}\ar@{^{(}-->}[dd]
			&&
			\mathfrak{z}(\bar{\mathfrak{m}})\ar@{==}[rr]\ar@{^{(}-->}[dd]
			&&
			\mathfrak{z}(\bar{\mathfrak{m}})\ar@{^{(}-->}[dd] \\
			\ker(q_{\Omega(M_{\ad},b_{\ad})})\ar[rr]^>>>>>>>>{\per_{\mathfrak{m}}}\ar@{^{(}->}[dd]\ar[ru]
			&&
			\mathfrak{z}(\mathfrak{m})\ar@{=}[rr]\ar@{^{(}->}[dd]\ar[ru]^{A}
			&&
			\mathfrak{z}(\mathfrak{m})\ar@{^{(}->}[dd]\ar[ru]^{A} \\
			&\widetilde\Omega(\bar M_{\ad},\bar b_{\ad})\ar@{-->}[rr]^<<<<<{f_{\bar{\mathfrak{m}}}}\ar@{-->}[dd]_>>>>>{q_{\Omega(\bar M_{\ad},\bar b_{\ad})}}
			&&
			\widehat\Omega(\bar M,\bar b)\ar@{^{(}-->}[rr]\ar@{-->}[dd]
			&&
			\widehat P(\bar M,\bar b)\ar@{-->}[rr]\ar@{-->}[dd]^<<<<<<<<{\chi_{\bar{\mathfrak{m}}}}
			&&
			(\bar M_{\ad},\bar b_{\ad})\ar@{==}[dd]\\
			\widetilde\Omega(M_{\ad},b_{\ad})\ar[rr]^>>>>>>>>{f_\mathfrak{m}}\ar[dd]_{q_{\Omega(M_{\ad},b_{\ad})}}\ar[ru]
			&&
			\widehat\Omega(M,b)\ar@{^{(}->}[rr]\ar[dd]\ar[ru]
			&&
			\widehat P(M,b)\ar[rr]\ar[dd]^<<<<<<<<{\chi_{\mathfrak{m}}}\ar[ru]
			&&
			(M_{\ad},b_{\ad})\ar@{=}[dd]\ar[ru]\\
			&\Omega(\bar M_{\ad},\bar b_{\ad})\ar@{==}[rr]
			&&
			\Omega(\bar M_{\ad},\bar b_{\ad})\ar@{^{(}-->}[rr]
			&&
			P(\bar M_{\ad},\bar b_{\ad})\ar@{-->}[rr]^<<<{\ev_1}
			&&
			(\bar M_{\ad},\bar b_{\ad}) \\
			\Omega(M_{\ad},b_{\ad})\ar@{=}[rr]\ar[ru]
			&&
			\Omega(M_{\ad},b_{\ad})\ar@{^{(}->}[rr]\ar[ru]
			&&
			P(M_{\ad},b_{\ad})\ar[rr]^-{\ev_1}\ar[ru]
			&&
			(M_{\ad},b_{\ad})\ar[ru]
		}
	\end{xy}$$
	In particular, we then have $A(\Pi(\mathfrak{m}))\subseteq \Pi(\bar{\mathfrak{m}})$.
\end{proposition}
\begin{proof}
	Cf.\ \cite[Lem.~III.3]{GN03} for the Lie algebra case.
	We shall refer to the different parts of the diagram by calling them cubes and faces of cubes.
	Since $A\colon\mathfrak{m}\rightarrow\bar{\mathfrak{m}}$ maps $\mathfrak{z}(\mathfrak{m})$ into  $\mathfrak{z}(\bar{\mathfrak{m}})$, it induces a morphism $A_{\ad}\colon \mathfrak{m}_{\ad}\rightarrow \bar{\mathfrak{m}}_{\ad}$ and hence morphisms $P(A_{\ad})\colon P(\mathfrak{m}_{\ad})\rightarrow P(\bar{\mathfrak{m}}_{\ad})$ and $\widehat P(A)\colon \widehat P(\mathfrak{m})\rightarrow \widehat P(\bar{\mathfrak{m}})$, which can be integrated to respective morphisms of pointed symmetric spaces. This gives us the rightmost cube, which is commutative, since it is easy to check that the corresponding cube on the level of Lie triple systems is commutative.
	
	By restricting the morphisms to the respective kernels, we obtain the three cubes on the right. Note that the induced morphism $\mathfrak{z}(\mathfrak{m})\rightarrow\mathfrak{z}(\bar{\mathfrak{m}})$ coincides with the suitable restriction of $A$, since $\Exp_{\mathfrak{z}(\mathfrak{m})}=\id_{\mathfrak{z}(\mathfrak{m})}$ and since the corresponding assertion on the level of Lie triple systems holds.
	The morphism $\Omega(M_{\ad},b_{\ad})\rightarrow \Omega(\bar M_{\ad},\bar b_{\ad})$ induces a morphism\linebreak $\widetilde\Omega(M_{\ad},b_{\ad})\rightarrow \widetilde\Omega(\bar M_{\ad},\bar b_{\ad})$ of the universal covering spaces, which can be restricted to the kernels of the covering morphisms. We obtain the two cubes on the left. The common face between these cubes is indeed commutative (and hence also the topmost face is), since it is easy to see that the corresponding diagram on the level of Lie triple systems is so.
\end{proof}
\begin{corollary} \label{cor:A(Pi(m))=Pi(bar m)}
	We have:
	\begin{enumerate}
		\item[\rm (1)] If $A\colon \mathfrak{m}\rightarrow \bar{\mathfrak{m}}$ is a morphism of Lie triple systems with $A(\mathfrak{z}(\mathfrak{m}))\subseteq \mathfrak{z}(\bar{\mathfrak{m}})$ for which the induced map $\ker(q_{\Omega(M_{\ad},b_{\ad})})\rightarrow \ker(q_{\Omega(\bar M_{\ad},\bar b_{\ad})})$ is surjective, then $A(\Pi(\mathfrak{m}))=\Pi(\bar{\mathfrak{m}})$.
		\item[\rm (2)] If $A\colon \mathfrak{m}\rightarrow \bar{\mathfrak{m}}$ is a surjective morphism of Lie triple systems with $A(\mathfrak{z}(\mathfrak{m}))= \mathfrak{z}(\bar{\mathfrak{m}})$ and $\ker(A)\subseteq \mathfrak{z}(\mathfrak{m})$, then $A(\Pi(\mathfrak{m}))=\Pi(\bar{\mathfrak{m}})$.
	\end{enumerate}
\end{corollary}
\begin{proof}
	Cf.\ \cite[Cor.~III.4]{GN03} for the Lie algebra case.
	
	(1) This follows immediately by Proposition~\ref{prop:functorialityOfPeriodSpace}.
	
	(2) The induced morphism $A_{\ad}\colon \mathfrak{m}_{\ad}\rightarrow \bar{\mathfrak{m}}_{\ad}$ is not only surjective, but also injective and hence an isomorphism, since we have $A^{-1}(\mathfrak{z}(\bar{\mathfrak{m}}))=\mathfrak{z}(\mathfrak{m})+\ker(A)=\mathfrak{z}(\mathfrak{m})$. Therefore, also the induced map $(M_{\ad},b_{\ad})\rightarrow (\bar M_{\ad},\bar b_{\ad})$ is an isomorphism. The assertion follows from (1).
\end{proof}
\begin{remark}
	We can consider $\Pi\colon \mathfrak{m}\mapsto \Pi(\mathfrak{m})$ as a covariant functor from the category of Lie triple systems where the morphisms are required to map center to center to the category of abelian topological groups.
\end{remark}
\begin{remark} \label{rem:periodSpaceOfDirectProduct}
	Given the direct product $\mathfrak{m}_1\times \mathfrak{m}_2$ of two Lie triple systems, one observes that $\Pi(\mathfrak{m}_1\times \mathfrak{m}_2)=\Pi(\mathfrak{m}_1)\times \Pi(\mathfrak{m}_2)$ by following the construction of the period group.
\end{remark}
\begin{lemma} \label{lem:im(f_m)IntegralSubspace}
	The image $\im(f_\mathfrak{m})$ of $f_\mathfrak{m}$ is the connected integral subspace of $\widehat\Omega(M,b)$ (and hence of $\widehat P(M,b)$) with Lie triple system
	$$\Omega(\mathfrak{m}_{\ad})\ \hookrightarrow\ \Omega(\mathfrak{m}_{\ad})\times \mathfrak{z}(\mathfrak{m}) \ =\ \widehat\Omega(\mathfrak{m}) \ \subseteq\ \widehat P(\mathfrak{m}).$$
\end{lemma}
\begin{proof}
	In the light of Proposition~\ref{prop:integralSubreflectionSpace}, the assertion follows by
	\begin{align*}
		\im(f_\mathfrak{m}) \ &=\  f_\mathfrak{m}\big(\big<\im(\Exp_{\widetilde\Omega(M_{\ad},b_{\ad})})\big>\big) \ =\ \big< \Exp_{\widehat\Omega(M,b)}\big(\im(\Lts(f_\mathfrak{m}))\big)\big> \\
		&=\ \big< \Exp_{\widehat\Omega(M,b)}(\Omega(\mathfrak{m_{\ad}})\times \{0\})\big>. \qedhere
	\end{align*}
\end{proof}
\begin{lemma} \label{lem:im(f_m)subsymSpaceIff...}
	The integral subspace $\im(f_\mathfrak{m})\leq \widehat\Omega(M,b)$ is actually a symmetric subspace if and only if the period group $\Pi(\mathfrak{m})$ is discrete.
\end{lemma}
\begin{proof}
	By Proposition~\ref{prop:integralSubSpace=subSymSpace}, $\im(f_\mathfrak{m})$ is a symmetric subspace if and only if there exists a 0-neighborhood $W\subseteq\mathfrak{z}(\mathfrak{m})$ such that $\im(f_\mathfrak{m})\cap\Exp_{\widehat\Omega(M,b)}(W)=\{0\}$. Since the symmetric space $\mathfrak{z}(\mathfrak{m})$ is a symmetric subspace of $\widehat\Omega(M,b)$ with Lie algebra $\mathfrak{z}(\mathfrak{m})\hookrightarrow \widehat\Omega(\mathfrak{m})$ (cf.\ (\ref{eqn:restrictionOfChi^m})), we have $\Exp_{\widehat\Omega(M,b)}(W)=\Exp_{\mathfrak{z}(\mathfrak{m})}(W)=W$. Together with $\im(f_\mathfrak{m})\cap W=\im(f_\mathfrak{m})\cap\mathfrak{z}(\mathfrak{m})\cap W = \Pi(\mathfrak{m})\cap W$, we obtain the assertion.
\end{proof}
\begin{theorem}[Integrability Criterion for Lie triple systems] \label{th:integrabilityCriterion}
	Let $\mathfrak{m}$ be a Banach--Lie triple system. Considering some symmetric Banach--Lie algebra $(\mathfrak{g},\theta)$ with $\mathfrak{g}_-=\mathfrak{m}$ and $\mathfrak{z}(\mathfrak{g})=\mathfrak{z}(\mathfrak{m})$ (e.g.\ the standard embedding $S(\mathfrak{m})$ of $\mathfrak{m}$), the following are equivalent:
	\begin{enumerate}
		\item[\rm (a)] The Lie triple system $\mathfrak{m}$ is integrable to a pointed symmetric space.
		\item[\rm (b)] The Lie algebra $\mathfrak{g}$ is integrable to a Lie group.
		\item[\rm (c)] The period group $\Pi(\mathfrak{m})$ is discrete.
		\item[\rm (d)] The period group $\Pi(\mathfrak{g})$ is discrete.
	\end{enumerate}
\end{theorem}
\begin{proof}
	The equivalence (b)$\Leftrightarrow$(d) holds by \cite[Th.~III.7]{GN03}. The implications (b)$\Rightarrow$(a) and (c)$\Rightarrow$(d) follow by Lemma~\ref{lem:integrableSymLieAlgebra} (and Remark~\ref{rem:integrabilityOfSymLieAlgebra}) and Proposition~\ref{prop:Pi(g)inPi(m)}, respectively. It remains to show the implication (a)$\Rightarrow$(c). For this, let $(M,b)$ be a pointed symmetric space with Lie triple system $\mathfrak{m}$. By the Integrability Theorem, we can integrate the natural morphism $\widehat P(\mathfrak{m})\rightarrow \mathfrak{m}$, $(\gamma,x)\mapsto x$ of Lie triple systems to a morphism	$p\colon\widehat P(M,b)\rightarrow (M,b)$ of pointed symmetric spaces. By Proposition~\ref{prop:kernelOfMorphismOfPointedSymSpaces}, its kernel $\ker(p)\unlhd\widehat P(M,b)$ is a symmetric subspace with Lie triple system $$\Lts(\ker(p)) \ =\ \ker(\Lts(p)) \ =\ \Omega(\mathfrak{m}_{\ad}) \ \hookrightarrow \ \Omega(\mathfrak{m}_{\ad})\times \mathfrak{z}(\mathfrak{m}) \ =\ \widehat\Omega(\mathfrak{m}) \ \subseteq\ \widehat P(\mathfrak{m}).$$
	Therefore, the basic connected component $(\ker(p))_0$ coincides with $\im(f_\mathfrak{m})$ by Lemma~\ref{lem:im(f_m)IntegralSubspace} and Proposition~\ref{prop:integralSubreflectionSpace}, so that (c) follows by Lemma~\ref{lem:im(f_m)subsymSpaceIff...}.
\end{proof}
\begin{remark}
	We shall see that the period group of a finite-dimensional Lie triple system is trivial (cf.\ Remark~\ref{rem:periodSpaceOfFiniteDimLts}).
\end{remark}
\begin{remark} \label{rem:ker(p)IsConnected}
	For later purposes, we need to know that the kernel $\ker(p)$ of the morphism $p\colon\widehat P(M,b)\rightarrow (M,b)$ (cf.\ the proof) is connected if $M$ is 1-connected. To see this, let $(G,\sigma)$ be a 1-connected symmetric Lie group with Lie algebra $(\mathfrak{g},\theta)$. Let $p_{(\mathfrak{g},\theta)}\colon \widehat P(G,\sigma)\rightarrow (G,\sigma)$ be the morphism of symmetric Lie groups with $L(p_{(\mathfrak{g},\theta)})\colon \widehat P(\mathfrak{g},\theta)\rightarrow (\mathfrak{g},\theta)$, $(\gamma,x)\mapsto x$. It is surjective, since $G$ is connected and $L(p_{(\mathfrak{g},\theta)})$ is surjective. Thus we have $G\cong \widehat P(G)/\ker(p_{(\mathfrak{g},\theta)})$ (cf.\ \cite[Th.~II.2]{GN03}), so that $\ker(p_{(\mathfrak{g},\theta)})$ is connected, the group $G$  being 1-connected (cf.\ \cite[p.~10]{GN03}).
	
	By Example~\ref{ex:typicalExactSequence}, we have an exact sequence
	\begin{multline} \label{eqn:sequence_p_(g,theta)}
		\eins \rightarrow (\ker(p_{(\mathfrak{g},\theta)}),\widehat P(\sigma)|_{\ker(p_{(\mathfrak{g},\theta)})},(\widehat P(G)^{\widehat P(\sigma)})_0 \cap  \ker(p_{(\mathfrak{g},\theta)}))\hookrightarrow (\widehat P(G,\sigma),(\widehat P(G)^{\widehat P(\sigma)})_0)\\
		\rightarrow (G,\sigma, (G^\sigma)_0)\rightarrow \eins
	\end{multline}
	of symmetric pairs:
	Indeed, to see that $p_{(\mathfrak{g},\theta)}((\widehat P(G)^{\widehat P(\sigma)})_0) = (G^\sigma)_0$, it suffices to verify the surjectivity of $\widehat P(\mathfrak{g})_+\rightarrow \mathfrak{g}_+$ (cf.\ p.~\pageref{page:L(f)(g1_+)=g2_+}), which is clear, since $(\mathfrak{g}_{\ad})_+=\mathfrak{g}_+$ entails $\widehat P(\mathfrak{g})_+ = \{(\gamma,x)\in P(\mathfrak{g}_+)\times \mathfrak{g}_+\colon \gamma(1)=x\}$.
	
	Applying the functor $\Sym$ to (\ref{eqn:sequence_p_(g,theta)}) leads to the exact sequence
	$$o \rightarrow \ker(p_{(\mathfrak{g},\theta)})/((\widehat P(G)^{\widehat P(\sigma)})_0 \cap  \ker(p_{(\mathfrak{g},\theta)}))\hookrightarrow \widehat P(G)/(\widehat P(G)^{\widehat P(\sigma)})_0\rightarrow G/(G^\sigma)_0\rightarrow o$$
	of pointed symmetric spaces where the injection is a topological embedding (cf.\ Proposition~\ref{prop:exactnessOfSym}).
	
	Since $G/(G^\sigma)_0$ is 1-connected
	with $\Lts(G/(G^\sigma)_0)=\mathfrak{g}_-=\mathfrak{m}$, it can be identified with $(M,b)$. With $\widehat P(M,b)=\widehat P(G)/(\widehat P(G)^{\widehat P(\sigma)})_0$ (cf.\ Proposition~\ref{prop:chi_m}), the morphisms $\Sym(p_{(\mathfrak{g},\theta)})$ and $p$ can be identified, since $\Lts(\Sym(p_{(\mathfrak{g},\theta)}))=\calLts(L(p_{(\mathfrak{g},\theta)})) = \Lts(p)$ (cf.\ (\ref{eqn:LtsSym=calLtsL})). Consequently, the kernel $\ker(p)$ is isomorphic to $\ker(p_{(\mathfrak{g},\theta)})/((\widehat P(G)^{\widehat P(\sigma)})_0 \cap  \ker(p_{(\mathfrak{g},\theta)}))$ and is hence connected, the kernel $\ker(p_{(\mathfrak{g},\theta)})$ being connected.
\end{remark}
For some purposes, a generalization of Proposition~\ref{prop:Pi(g)inPi(m)} is useful that applies for symmetric Lie algebras $(\mathfrak{g},\theta)$ that not necessarily satisfy  $\mathfrak{z}(\mathfrak{g}) = \mathfrak{z}(\mathfrak{m})$.
\begin{proposition} \label{prop:Pi(g)_-inPi(m)}
	Let $\mathfrak{m}$ be a Lie triple system. Given a symmetric Lie algebra $(\mathfrak{g},\theta)$ with $\mathfrak{g}_-=\mathfrak{m}$ (so that $\mathfrak{z}(\mathfrak{g})_-\subseteq \mathfrak{z}(\mathfrak{m})$, cf.\ Section~\ref{sec:symLieAlgAndLieGrp}). Then the group $\Pi(\mathfrak{g})_-$ (cf.\ Remark~\ref{rem:periodGroupWithInvolution}) is a subset of the period group $\Pi(\mathfrak{m})$.
\end{proposition}
\begin{proof}
	We have $\mathfrak{g}_{\ad} = \mathfrak{g}_+/\mathfrak{z}(\mathfrak{g})_+ \oplus \mathfrak{m}/\mathfrak{z}(\mathfrak{g})_-$,
	so that $(\mathfrak{g}_{\ad})_-=\mathfrak{m}/\mathfrak{z}(\mathfrak{g})_-$. The natural map $(\mathfrak{g}_{\ad})_-\rightarrow \mathfrak{m}_{\ad}=\mathfrak{m}/\mathfrak{z}(\mathfrak{m})$ leads to a morphism $\Omega((\mathfrak{g}_{\ad})_-)\rightarrow\Omega(\mathfrak{m}_{\ad})$ as well as to one from $\widehat\Omega(\mathfrak{g})_-=\Omega((\mathfrak{g}_{\ad})_-)\times \mathfrak{z}(\mathfrak{g})_-$ to $\widehat\Omega(\mathfrak{m})=\Omega(\mathfrak{m}_{\ad})\times \mathfrak{z}(\mathfrak{m})$. We claim that they can be integrated to morphisms
	$$g_1\colon\widetilde\Omega(G_{\ad})/(\widetilde\Omega(G_{\ad})^{\widetilde\Omega(\sigma_{\ad})})_0\rightarrow \widetilde\Omega(M_{\ad},b_{\ad})$$
	and
	$$g_2\colon\widehat\Omega(G)/((\widehat P(G)^{\widehat P(\sigma)})_0 \cap \widehat\Omega(G)) \rightarrow \widehat\Omega(M,b),$$
	respectively. The former integration works by the Integrability Theorem and Proposition~\ref{prop:1-connectedQuotient}.
	To show the existence of $g_2$, we first consider the natural morphism from
	$$\widehat P(\mathfrak{g})_- =\{(\gamma,x)\in P((\mathfrak{g}_{\ad})_-)\times\mathfrak{g}_-\colon\gamma(1)=x+\mathfrak{z}(\mathfrak{g})_-\}$$
	to
	$$\widehat P(\mathfrak{m}) =\{(\gamma,x)\in P(\mathfrak{m}_{\ad})\times\mathfrak{m}\colon\gamma(1)=x+\mathfrak{z}(\mathfrak{m})\}$$
	as well as the morphism $(\mathfrak{g}_{\ad})_-\rightarrow \mathfrak{m}_{\ad}$. They can be integrated to morphisms
	$$g_3\colon \widehat P(G)/(\widehat P(G)^{\widehat P(\sigma)})_0 \rightarrow \widehat P(M,b)
	\quad\text{and}\quad
	g_4\colon G_{\ad}/(G_{\ad}^{\sigma_{\ad}})_0\rightarrow (M_{\ad},b_{\ad}),$$
	respectively, by the Integrability Theorem and Proposition~\ref{prop:1-connectedQuotient}.
	We obtain the diagram
	$$\begin{xy}
		\xymatrix{
			o\ar[r] & \widehat\Omega(M,b)\ar@{^{(}->}[r] & \widehat P(M,b)\ar[r] & (M_{\ad},b_{\ad})\ar[r] & o \\
			o\ar[r] & \Omega(G)/((\widehat P(G)^{\widehat P(\sigma)})_0 \cap \widehat\Omega(G))\ar@{^{(}->}[r] & \widehat P(G)/(\widehat P(G)^{\widehat P(\sigma)})_0\ar[r]\ar[u]_{g_3} & G_{\ad}/(G_{\ad}^{\sigma_{\ad}})_0\ar[r]\ar[u]_{g_4} & o
		}
	\end{xy}$$
	where the exact sequences are taken from the diagrams in Corollary~\ref{cor:diagramOfQuotients} and Remark~\ref{rem:diagramSymSpace}. It is commutative, since the corresponding diagram on the level of Lie triple systems is commutative. Thus $g_3$ restricts to the wanted morphism $g_2$.
		
	We have $f_{\mathfrak{m}}\circ g_1 = g_2\circ \Sym(f_{(\mathfrak{g},\theta)})$, since the corresponding morphisms
	$$\Lts(f_{\mathfrak{m}}\circ g_1)\colon \Omega((\mathfrak{g}_{\ad})_-)\rightarrow\Omega(\mathfrak{m}_{\ad}) \hookrightarrow \widehat\Omega(\mathfrak{m})$$
	and
	$$\Lts(g_2\circ \Sym(f_{(\mathfrak{g},\theta)}))\colon \Omega((\mathfrak{g}_{\ad})_-)\hookrightarrow \widehat\Omega(\mathfrak{g})_- \rightarrow \widehat\Omega(\mathfrak{m})$$
	are equal.
	
	We obtain the commutative diagram
	$$\begin{xy}
		\xymatrixcolsep{-1pc} 
		\xymatrix{
			&&&&&	\mathfrak{z}(\mathfrak{m})\ar@{^{(}->}[ddd] \\
			&&&& \mathfrak{z}(\mathfrak{g})/\mathfrak{z}(\mathfrak{g})_+\cong \mathfrak{z}(\mathfrak{g})_-\ar@{^{(}->}[ddd]\ar@{^{(}->}[ru]^{\iota} \\
			&&&
			\mathfrak{z}(\mathfrak{g},\theta)\ar@{^{(}->}[ddd]\ar[ru]^{q_1} \\
			&&\widetilde\Omega(M_{\ad},b_{\ad})\ar[rrr]^{f_\mathfrak{m}}
			&&& \widehat\Omega(M,b) \\
			&\widetilde\Omega(G_{\ad})/(\widetilde\Omega(G_{\ad})^{\widetilde\Omega(\sigma_{\ad})})_0\ar[rrr]^-{\Sym(f_{(\mathfrak{g},\theta)})}\ar[ru]^{g_1}
			&&& \widehat\Omega(G)/((\widehat P(G)^{\widehat P(\sigma)})_0 \cap \widehat\Omega(G))\ar[ru]^{g_2} \\
			\widetilde\Omega(G_{\ad},\sigma_{\ad})\ar[rrr]^-{f_{(\mathfrak{g},\theta)}}\ar[ru]^{q_2}
			&&& \widehat\Omega(G,\sigma)\ar[ru]^{q_3} \\
		}
	\end{xy}$$
	where $\iota\colon \mathfrak{z}(\mathfrak{g})_-\hookrightarrow \mathfrak{z}(\mathfrak{m})$ is the inclusion morphism and where the maps $q_i$ with $i=1,2,3$ are the quotient maps given by the corresponding symmetric pairs. (Note that the commutativity of the part of the diagram containing $\iota$ and $g_2$ follows by the corresponding commutative diagram for the level of Lie triple systems.)
	
%
	Similarly as in the proof of Proposition~\ref{prop:Pi(g)inPi(m)}, we have
	$$(\iota\circ q_1)(\Pi(\mathfrak{g})) \ \subseteq\ (g_2\circ q_3)(\im(f_{(\mathfrak{g},\theta)}))\cap\mathfrak{z}(\mathfrak{m}) \ \subseteq\ \im(f_{\mathfrak{m}})\cap\mathfrak{z}(\mathfrak{m}) \ =\ \Pi(\mathfrak{m}).$$
	From $\Pi(\mathfrak{g})_+ \oplus \Pi(\mathfrak{g})_- \leq \Pi(\mathfrak{g})$ (cf.\ Remark~\ref{rem:periodGroupWithInvolution}), we deduce that $\Pi(\mathfrak{g})_- \leq q_1(\Pi(\mathfrak{g}))$, entailing that $\Pi(\mathfrak{g})_- \subseteq \Pi(\mathfrak{m})$.
\end{proof}
Let $G$ be a Banach--Lie group with Lie algebra $\mathfrak{g}$ and let $x \in \mathfrak{z}(\mathfrak{g})$ and $y \in \mathfrak{g}$. Then $\exp(x)$ is in the center $Z(G)$ of $G$, so that $\exp(x+y)=\exp(x)\exp(y)$ can be observed by the Trotter Product Formula (cf.\ \cite[Th.~IV.2]{Nee04}).
In particular, the restriction $\exp|_{\mathfrak{z}(\mathfrak{g})}$ of the exponential map is a smooth homomorphism of Lie groups. An analogous statement for symmetric spaces is given by the following lemma:
\begin{lemma}\label{lem:Exp|_z(m)}
	Let $(M,b)$ be a pointed symmetric space with Lie triple system $\mathfrak{m}$. Then, for all $x\in \mathfrak{z}(\mathfrak{m})$ and $y\in \mathfrak{m}$, we have
	$$\Exp_{(M,b)}(2x-y) \ =\ \Exp_{(M,b)}(x)\cdot \Exp_{(M,b)}(y).$$
	In particular, the restriction $\Exp_{(M,b)}|_{\mathfrak{z}(\mathfrak{m})}$ of the exponential map is a morphism of pointed symmetric spaces.
\end{lemma}
\begin{proof}
	W.l.o.g., we assume $M$ to be connected, so that we can identify it with the quotient $G^\prime(M,b)/G^\prime(M,b)_b$  of the symmetric pair
	$(G^\prime(M,b),\sigma^\prime,G^\prime(M,b)_b)$ (cf.\ Section~\ref{sec:AutOfConnectedSymSpace}).
	Denoting the quotient map by $q\colon G^\prime(M,b) \rightarrow G^\prime(M,b)/G^\prime(M,b)_b$, 
	the exponential maps $\exp\colon \mathfrak{g}^\prime(M,b)\rightarrow G^\prime(M,b)$ and $\Exp\colon \mathfrak{g}^\prime(M,b)_-\rightarrow G^\prime(M,b)/G^\prime(M,b)_b$ satisfy $\Exp=q\circ \exp|_{\mathfrak{g}^\prime(M,b)_-}$ (cf.\ (\ref{eqn:Exp=q exp})).
	
	The symmetric Lie algebra $\mathfrak{g}^\prime(M,b)$ can be identified with the standard embedding of $\mathfrak{g}^\prime(M,b)_-$ (cf.\ Proposition~\ref{prop:g'(M,b)IsomorphicS(m)}), so that its center $\mathfrak{z}(\mathfrak{g}^\prime(M,b))$ coincides with $\mathfrak{z}(\mathfrak{g}^\prime(M,b)_-)$ (cf.\ Section~\ref{sec:symLieAlgAndLieGrp}).
	Therefore, given any $x\in \mathfrak{z}(\mathfrak{g}^\prime(M,b)_-)$ and $y\in\mathfrak{g}^\prime(M,b)_-$, we have $\exp(2x-y)=\exp(2x)\exp(-y)$. Because of $\sigma^\prime\circ\exp=\exp\circ L(\sigma^\prime)$ and $L(\sigma^\prime)|_{\mathfrak{g}^\prime(M,b)_-}=-\id_{\mathfrak{g}^\prime(M,b)_-}$, we have $\sigma^\prime(\exp(x))=\exp(-x)=(\exp(x))^{-1}$ and $\sigma^\prime(\exp(-y))=\exp(y)$, so that we obtain
	\begin{align*}
		\Exp(2x-y) \ &=\ q(\exp(x)\exp(x)\exp(-y)) \ =\ q(\exp(x)\sigma^\prime(\exp(x))^{-1}\sigma^\prime(\exp(y))) \\
		&=\ \Exp(x)\cdot\Exp(y). \qedhere
	\end{align*}
\end{proof}
\begin{proposition} \label{prop:periodSpace=kernel}
	Let $(M,b)$ be a pointed 1-connected symmetric space with Lie triple system $\mathfrak{m}$. Then the period group is given by $\Pi(\mathfrak{m}) = \ker(\Exp_{(M,b)}|_{\mathfrak{z}(\mathfrak{m})})$.
\end{proposition}
\begin{proof}
	As in the proof of Theorem~\ref{th:integrabilityCriterion}, let $p\colon \widehat P(M,b)\rightarrow (M,b)$ be the morphism of pointed symmetric spaces with $\Lts(p)\colon\widehat P(\mathfrak{m})\rightarrow \mathfrak{m}$, $(\gamma,x)\mapsto x$. Denoting by $\iota\colon \mathfrak{z}(\mathfrak{m})\hookrightarrow \widehat P(M,b)$ the inclusion morphism given in the diagram of Remark~\ref{rem:diagramSymSpace}, we claim that $\Exp_{(M,b)}|_{\mathfrak{z}(\mathfrak{m})} = p\circ \iota$. To see this, it suffices to show that $\Lts(\Exp_{(M,b)}|_{\mathfrak{z}(\mathfrak{m})}) =\Lts( p\circ \iota)$, i.e., that $(\id_{\mathfrak{m}})|_{\mathfrak{z}(\mathfrak{m})} = \Lts(p)\circ \Lts(\iota)$, but the latter is true, since $\Lts(\iota)$ is the inclusion morphism $\mathfrak{z}(\mathfrak{m})\hookrightarrow \widehat P(\mathfrak{m})$ given in the diagram of Remark~\ref{rem:diagramOfLts}.
	
	The kernel $\ker(p \circ \iota)$ of $\Exp_{(M,b)}|_{\mathfrak{z}(\mathfrak{m})}$ is given by $\ker(p)\cap \mathfrak{z}(\mathfrak{m})$.
	Since $\ker(p)$ is connected (cf.\ Remark~\ref{rem:ker(p)IsConnected}), we have $\ker(p)=\im(f_\mathfrak{m})$ (cf.\ the proof of Theorem~\ref{th:integrabilityCriterion}), so that we obtain
	\[ \Pi(\mathfrak{m}) \ =\ \im(f_\mathfrak{m})\cap\mathfrak{z}(\mathfrak{m}) \ =\ \ker(p)\cap\mathfrak{z}(\mathfrak{m}) \ =\ \ker(\Exp_{(M,b)}|_{\mathfrak{z}(\mathfrak{m})}). \qedhere\]
\end{proof}
\begin{corollary} \label{cor:periodSpace&Group&Kernel}
	Considering some 1-connected symmetric Lie group $(G,\sigma)$ whose Lie algebra $(\mathfrak{g},\theta)$ satisfies $\mathfrak{g}_-=\mathfrak{m}$ and $\mathfrak{z}(\mathfrak{g})=\mathfrak{z}(\mathfrak{m})$, we have $\Pi(\mathfrak{m})=(\exp_G|_{\mathfrak{z}(\mathfrak{g})})^{-1}((G^\sigma)_0)$.
\end{corollary}	
\begin{proof}
	By the Integrability Theorem and Proposition~\ref{prop:1-connectedQuotient}, we can identify $(M,b)$ with\linebreak $G/(G^\sigma)_0$, since we have $\Lts(G/(G^\sigma)_0)=\mathfrak{g}_-=\mathfrak{m}$. The exponential map $\Exp_{G/(G^\sigma)_0}$ is given by $q\circ \exp_G|_{\mathfrak{m}}$ with the quotient map $q\colon G\rightarrow G/(G^\sigma)_0$, so that we obtain
	\[\Pi(\mathfrak{m})\ =\ \ker(\Exp_{G/(G^\sigma)_0}|_{\mathfrak{z}(\mathfrak{g})}) \ =\ (\exp_G|_{\mathfrak{z}(\mathfrak{g})})^{-1}((G^\sigma)_0). \qedhere\]
\end{proof}
\begin{example}[Period group of a loop triple system]
	Let $(M,b)$ be a pointed 2-connected symmetric space with Lie triple system $\mathfrak{m}$, i.e., $M$ is connected and its homotopy groups $\pi_1(M)$ and $\pi_2(M)$ are trivial. Then the loop space $\Omega(M,b)$ is 1-connected (cf.\ \cite[Cor.~VII.4.4]{Bre93}), so that the period group of the loop triple system $\Omega(\mathfrak{m})$ of $\mathfrak{m}$ is given by
	$\Pi(\Omega(\mathfrak{m}))=\ker(\Exp_{\Omega(M,b)}|_{\mathfrak{z}(\Omega(\mathfrak{m}))})$
	(cf.\ Proposition~\ref{prop:periodSpace=kernel}). Since the Lie triple bracket on $\Omega(\mathfrak{m})$ is defined pointwise and since for each $x\in\mathfrak{m}$ and $t\in [0,1]$, there exists a loop $\gamma\in\Omega(\mathfrak{m})$ with $\gamma(t)=x$, the center $\mathfrak{z}(\Omega(\mathfrak{m}))$ is given by $\Omega(\mathfrak{z}(\mathfrak{m}))$. Therefore, we obtain
	$$\Pi(\Omega(\mathfrak{m})) \ =\ \ker(\Omega(\Exp_{(M,b)})|_{\Omega(\mathfrak{z}(\mathfrak{m}))}) \ =\ \Omega(\ker(\Exp_{(M,b)}|_{\mathfrak{z}(\mathfrak{m})})) \ =\ \Omega(\Pi(\mathfrak{m}))$$
	(cf.\ Proposition~\ref{prop:periodSpace=kernel}). By Theorem~\ref{th:integrabilityCriterion}, the period group $\Pi(\mathfrak{m})$ is discrete. Hence it follows that $\Pi(\Omega(\mathfrak{m}))=\{0\}$.

\end{example}
We now construct an example of a Lie triple system with one-dimensional center and period group isomorphic to $\ZZ$.
\begin{example} \label{ex:u(H)}
	Let $(H,I)$ be an infinite-dimensional\footnote{The infinite-dimensionality is not required until Kuiper's Theorem is applied.} complex Hilbert space $H$ with a conjugation $I$, i.e., an antilinear isometry satisfying $I^2=\id_{H}$. The unitary group
	$U(H):=\{g\in \GL(H)\colon g^\ast=g^{-1}\}$
	(where $\GL(H)$ is related to $\CC$) is a real Lie subgroup of $\GL(H)$ and can be endowed with the involutive automorphism
	$\sigma\colon U(H)\rightarrow U(H)$, $g\mapsto IgI$.
	Its Lie algebra
	$\mathfrak{u}(H):=\{X\in B(H)\colon X^\ast=-X\}$
	of skew-hermitian operators is endowed with the involutive automorphism
	$L(\sigma)\colon \mathfrak{u}(H)\rightarrow \mathfrak{u}(H)$, $X\mapsto IXI$.
	
	Considering the real form $H_\RR:=\{v\in H\colon I(v)=v\}$ of the complex Hilbert space $H$, there is a canonical closed embedding
	$\varepsilon\colon B(H_\RR)\hookrightarrow B(H),\ X\mapsto X_\CC$
	satisfying $X_\CC|_{H_\RR}=X$ and $(X_\CC|_{iH_\RR})(iv)=iX(v)$ (for all $v\in H_\RR$), whose image is $\im(\varepsilon)=\{X\in B(H)\colon XI = IX\}$ and that maps $\GL(H_\RR)$ onto $\{g\in \GL(H)\colon gI = Ig\}$.
	
	Then the fixed point group $U(H)^\sigma$ is given by
	$$U(H)^\sigma \ =\ U(H) \cap \varepsilon(\GL(H_\RR)) \ =\ \varepsilon(O(H_\RR)) \ \cong\ O(H_\RR) :=\{g\in\GL(H_\RR)\colon g^\top = g^{-1}\}$$
	and its Lie algebra $\mathfrak{u}(H)_+$ by
	$$\mathfrak{u}(H)_+ \ =\ \mathfrak{u}(H) \cap \im(\varepsilon) \ =\ \varepsilon(\mathfrak{o}(H_\RR)) \ \cong\ \mathfrak{o}(H_\RR) :=\{X\in B(H_\RR)\colon X^\top = -X\}.$$
	Further we have
	$$\mathfrak{u}(H)_- \ =\ \mathfrak{u}(H) \cap i\im(\varepsilon) \ =\ i\varepsilon(\Sym(H_\RR)) \quad \text{with } \Sym(H_\RR) :=\{X\in B(H_\RR)\colon X^\top = X\}.$$
	
	It is well known that the center $\mathfrak{z}(\mathfrak{u}(H))$ of the Lie algebra $\mathfrak{u}(H)$ is given by $\RR i\id_H$. We claim that it coincides with the center $\mathfrak{z}(\mathfrak{u}(H)_-)$ of the Lie triple system $\mathfrak{u}(H)_-$.
	From $\mathfrak{z}(\mathfrak{u}(H))=\RR i\id_H\subseteq \mathfrak{u}(H)_-$, we deduce that $\RR i\id_H=\mathfrak{z}(\mathfrak{u}(H))_- \subseteq \mathfrak{z}(\mathfrak{u}(H)_-)$ (cf.\ Section~\ref{sec:symLieAlgAndLieGrp}), so that it remains to show that $\RR i\id_H \supseteq \mathfrak{z}(\mathfrak{u}(H)_-)$.
	
	Given any $i\varepsilon(X)\in \mathfrak{z}(\mathfrak{u}(H)_-)$ (with $X\in\Sym(H_\RR)$), we have $[[i\varepsilon(X),i\varepsilon(Y)],i\varepsilon(Z)]=0$ for all $Y,Z\in\Sym(H_\RR)$, entailing $[[X,Y],Z]=0$ in the Lie algebra $B(H_\RR)$. For each $v\in H_\RR\backslash \{0\}$, we define $T_v\in \Sym(H_\RR)$ by $T_v(x\oplus y):=x\oplus 0$ with respect to the decomposition $H_\RR=\RR v \oplus v^\bot$ and note that $T_v^2 = T_v$.
	Because of
	\begin{equation} \label{eqn:[A,T_v](v)}
		[A,T_v](v) = A(v) - T_v(A(v)) \quad \text{for all $A\in B(H_\RR)$,}
	\end{equation}
	we have
	\begin{eqnarray*}
		[[X,T_v],T_v](v) &=& [X,T_v](v) - T_v([X,T_v](v)) \\
		&=&  X(v) - T_v(X(v)) - T_v(X(v)) + (T_vT_vX)(v)
		\ =\ X(v) - T_v(X(v)),
	\end{eqnarray*}
	and hence $X(v)=T_v(X(v))\in\RR v$. Thus, every vector $v\in H_{\RR}$ is an eigenvector of $X$, so that we have $X=\lambda\id_{H_{\RR}}$ for some $\lambda\in\RR$ by a simple argument of linear algebra. It follows that $i\varepsilon(X)=\lambda i \id_H$, i.e., $\RR i\id_H \supseteq \mathfrak{z}(\mathfrak{u}(H)_-)$.
		
	Knowing now that
	$$\mathfrak{z}(\mathfrak{u}(H)) \ =\ \RR i\id_H \ =\ \mathfrak{z}(\mathfrak{u}(H)_-)$$
	and taking into account that $U(H)$ is 1-connected (actually contractible) by Kuiper's Theorem (cf.\ \cite{Nee02Classical}), we can apply Corollary~\ref{cor:periodSpace&Group&Kernel} and obtain
	$$\Pi(\mathfrak{u}(H)_-) \ =\ (\exp_{U(H)}|_{\mathfrak{z}(\mathfrak{u}(H))})^{-1}((U(H)^\sigma)_0) \ =\ (\exp_{U(H)}|_{\RR i\id_H})^{-1}(U(H)^\sigma),$$
	since $U(H)^\sigma\cong O(H_\RR)$ is connected (by Kuiper's Theorem). Since $\exp_{U(H)}$ is a restriction of the exponential map of $\GL(H)$ that is given by $\exp\colon B(H)\rightarrow \GL(H)$, $X\mapsto \sum_{n=0}^\infty \frac{X^n}{n!}$ (cf.\ \cite[Prop.~IV.9]{Nee04}), we have $\exp(t i\id_H)= e^{it}\id_H$ for all $t\in \RR$. The condition\linebreak $e^{it}\id_H\in U(H)^\sigma$ with  $U(H)^\sigma = U(H)\cap \varepsilon(\GL(H_\RR))$ is satisfied if and only if $e^{it}\id_HI=Ie^{it}\id_H$, i.e., $Ie^{-it}\id_H=Ie^{it}\id_H$. This is equivalent to $1=e^{2it}$, so that the period group of $\mathfrak{u}(H)_-$ is given by
	$$\Pi(\mathfrak{u}(H)_-) \ =\ \ZZ \pi i \id_H \ \cong\ \ZZ.$$
	To compare this with the period group of $\mathfrak{u}(H)$, we compute
	\begin{eqnarray*}
		\Pi(\mathfrak{u}(H)) &=& \ker(\exp_{U(H)}|_{\RR i\id_H}) \ =\ \{ti\id_H \in \RR i \id_H\colon e^{it}=1\} \\
		&=& \ZZ 2\pi i\id_H \ =\ 2\Pi(\mathfrak{u}(H)_-) \ \cong\ \ZZ
	\end{eqnarray*}
	(cf.\ Remark~\ref{rem:periodGroup=kernel}).
\end{example}
%
%
%
%
%
%
\subsection{Constructing Non-Integrable Quotients}
In this subsection, we construct non-integrable Lie triple systems by finding suitable quotients of integrable Lie triple systems with non-trivial period group.
\begin{lemma} \label{lem:sequence2x_n-y_n}
	Let $\mathfrak{m}$ be a Lie triple system, $\mathfrak{n}\unlhd \mathfrak{m}$ a closed ideal and $q\colon \mathfrak{m}\rightarrow \mathfrak{m}/\mathfrak{n}$ the corresponding quotient morphism. Then the following are equivalent:
	\begin{enumerate}
		\item[\rm (a)] The subgroup $q(\Pi(\mathfrak{m}))\leq \mathfrak{m}/\mathfrak{n}$ is not discrete.
		\item[\rm (b)] There exist sequences $(x_n)_{n\in\NN}$ in $\Pi(\mathfrak{m})\backslash\mathfrak{n}$ and $(y_n)_{n\in\NN}$ in $\mathfrak{n}$ with $\lim_{n\to\infty} 2x_n-y_n = 0$ in $\mathfrak{m}$.
	\end{enumerate}
\end{lemma}
\begin{proof}
	Because of $2(\Pi(\mathfrak{m})\backslash\mathfrak{n})-\mathfrak{n} =(2\Pi(\mathfrak{m})+\mathfrak{n})\backslash\mathfrak{n}$, Condition (b) is equivalent to the existence of a sequence $(z_n)_{n\in\NN}$ in $(2\Pi(\mathfrak{m})+\mathfrak{n})\backslash\mathfrak{n}$ with $\lim_{n\to\infty} z_n = 0$. The image of such a sequence $(z_n)_{n\in\NN}$ under $q$ is a null sequence lying in $q(2\Pi(\mathfrak{m}))\backslash \{0\}$. Conversely, by Michael's Selection Theorem (cf.\ \cite{Mic56} and also \cite{BG52}), every null sequence lying in $q(2\Pi(\mathfrak{m}))\backslash \{0\}$ can be obtained as the image of such a sequence $(z_n)_{n\in\NN}$. Thus, Condition (b) is equivalent to the existence of a sequence $(w_n)_{n\in\NN}$ in $2q(\Pi(\mathfrak{m}))\backslash \{0\}$ with $\lim_{n\to\infty} w_n = 0$, hence to Condition (a).
\end{proof}
\begin{theorem}[Integrability Criterion for quotients]
\label{th:necessaryCondForIntegrabOfQuot}
	Let $(M,b)$ be a pointed 1-connected symmetric space with Lie triple system $\mathfrak{m}$. Let $\mathfrak{n}\unlhd \mathfrak{m}$ be a closed ideal and $q\colon \mathfrak{m}\rightarrow \mathfrak{m}/\mathfrak{n}$ the corresponding quotient morphism. The conditions
	\begin{enumerate}
		\item[\rm (a)] The quotient $\mathfrak{m}/\mathfrak{n}$ is integrable.
		\item[\rm (b)] The normal integral subspace $N:=\langle\Exp_{(M,b)}(\mathfrak{n})\rangle \unlhd M$ is a closed symmetric subspace.
	\end{enumerate}
	are equivalent and imply the condition
	\begin{enumerate}
		\item[\rm (c)] The subgroup $q(\Pi(\mathfrak{m}))\leq \mathfrak{m}/\mathfrak{n}$ is discrete.
	\end{enumerate}
\end{theorem}
\begin{proof}
	(a)$\Rightarrow$(b): Let $(Q,b_Q)$ be a pointed 1-connected symmetric space with $\Lts(Q,b_Q)=\mathfrak{m}/\mathfrak{n}$. If we consider the unique morphism $q\colon (M,b)\rightarrow (Q,b_Q)$ for which $\Lts(q)\colon \mathfrak{m}\rightarrow \mathfrak{m}/\mathfrak{n}$ is the quotient morphism, then $\ker(q)$ is a closed symmetric subspace of $M$ with $\Lts(\ker(q))=\mathfrak{n}$ (cf.\ Proposition~\ref{prop:kernelOfMorphismOfPointedSymSpaces}). It follows that its basic connected component $(\ker(q))_0$ coincides with $N$, so that $N$ is a symmetric subspace by uniqueness (cf.\ Proposition~\ref{prop:integralSubreflectionSpace}). Further, $N$ is closed in $M$, since $(\ker(q))_0$ is an open and hence closed subspace of $\ker(q)$ (cf.\ Section~\ref{sec:subspacesAndQuotients}).
	
	(b)$\Rightarrow$(a): By Theorem~\ref{th:quotients}, $M/N$ carries the structure of a pointed symmetric space with Lie triple system $\mathfrak{m}/\mathfrak{n}$.
	
	(b)$\Rightarrow$(c): To reach a contradiction, suppose that $q(\Pi(\mathfrak{m}))$ is not discrete. By Lemma~\ref{lem:sequence2x_n-y_n}, there then exist sequences $(x_n)_{n\in\NN}$ in $\Pi(\mathfrak{m})\backslash\mathfrak{n}$ and $(y_n)_{n\in\NN}$ in $\mathfrak{n}$ with $\lim_{n\to\infty} 2x_n-y_n = 0$ in $\mathfrak{m}$. By Proposition~\ref{prop:expChartOfSubsymSpace}, there is an open $0$-neighborhood $V\subseteq\mathfrak{m}$ such that $\Exp_{(M,b)}|_V$ is a diffeomorphism onto an open subset of $M$ and $\Exp_{(M,b)}(V\cap \mathfrak{n}) \ =\ \Exp_{(M,b)}(V) \cap N$. The sequence $(z_n)_{n\in\NN}$ with $z_n:=2x_n-y_n$ satisfies $z_n\notin \mathfrak{n}$ (for all $n\in\NN$), since $x_n\notin \mathfrak{n}$, but $y_n\in \mathfrak{n}$. On the other hand, because of $x_n \in\Pi(\mathfrak{m})$, we have
	$$\Exp_{(M,b)}(z_n) \ =\ \Exp_{(M,b)}(x_n)\cdot\Exp_{(M,b)}(y_n) \ \in\ b\cdot N \ =\ N $$
	by Lemma~\ref{lem:Exp|_z(m)} and Proposition~\ref{prop:periodSpace=kernel}. From $\lim_{n\to\infty} z_n = 0$, we deduce that there exists an $n_0\in\NN$ with $z_{n_0}\in V$, so that
	$$\Exp_{(M,b)}(z_{n_0}) \ \in\ \Exp_{(M,b)}(V)\cap N \ =\ \Exp_{(M,b)}(V\cap \mathfrak{n}).$$
	Hence it follows that $z_{n_0}\in \mathfrak{n}$, which leads to a contradiction.
\end{proof}
\begin{example}\label{ex:non-integrableQuotient}
	Let $\mathfrak{m}$ be an integrable Lie triple system with non-trivial period group $\Pi(\mathfrak{m})\neq \{0\}$ (cf., e.g., Example~\ref{ex:u(H)}). Considering some closed ideal $\mathfrak{n}\unlhd\mathfrak{m}$ with $\mathfrak{n}\cap \Pi(\mathfrak{m}) \neq \{0\}$ (e.g.\ $\mathfrak{n}:=\mathfrak{m}$ or $\mathfrak{n}:=\mathfrak{z}(\mathfrak{m})$), we shall show that the quotient $(\mathfrak{m}\times \mathfrak{m})/\{(x,\sqrt{2}x)\in\mathfrak{m}\times\mathfrak{m}\colon x\in\mathfrak{n}\}$ is not integrable by Theorem~\ref{th:necessaryCondForIntegrabOfQuot}.\footnote{Of course, $\sqrt{2}$ can be replaced by any other irrational number.}
	
	Given some $d\in (\mathfrak{n}\cap \Pi(\mathfrak{m}))\backslash \{0\}$, we have
	\begin{equation} \label{eqn:ZdxZd}
		\Pi(\mathfrak{m}\times\mathfrak{m}) \ =\ \Pi(\mathfrak{m})\times\Pi(\mathfrak{m}) \ \supseteq\ \langle d \rangle \times \langle d \rangle \ =\ \ZZ d \times \ZZ d
	\end{equation}
	(cf.\ Remark~\ref{rem:periodSpaceOfDirectProduct}).
	Let
	$q\colon \mathfrak{m}\times\mathfrak{m} \rightarrow (\mathfrak{m}\times\mathfrak{m}) / \{(x,\sqrt{2}x)\in\mathfrak{m}\times\mathfrak{m}\colon x\in\mathfrak{n}\}$
	be the quotient morphism. If $q(\Pi(\mathfrak{m}\times\mathfrak{m}))$ was discrete, then its preimage
	$$\big(\Pi(\mathfrak{m}\times\mathfrak{m}) + \{(x,\sqrt{2}x)\in\mathfrak{m}\times\mathfrak{m}\colon x\in\mathfrak{n}\}\big) \cap (\RR d \times \{0\})$$
	under the injective restriction $q|_{(\RR d \times \{0\})}$ would be discrete, too. The subset
	$$((\ZZ d \times \ZZ d) + \RR(d,\sqrt{2}d)) \cap (\RR d\times \{0\})$$
	(cf.\ (\ref{eqn:ZdxZd})) would then be discrete, which would contradict the non-discreteness of\linebreak $((\ZZ\times\ZZ) + \RR(1,\sqrt{2})) \cap (\RR \times \{0\})$, the former set arising as the image of the latter under the topological embedding $\RR\times\RR \hookrightarrow \mathfrak{m}\times\mathfrak{m}$, $(x,y)\mapsto (xd,yd)$.
	\end{example}
\begin{remark} \label{rem:periodSpaceOfFiniteDimLts}
	The period group of a finite-dimensional Lie triple system is trivial, because otherwise Example~\ref{ex:non-integrableQuotient} would give us a non-integrable finite-dimensional Lie triple system (contradicting Corollary~\ref{cor:finite-dimLtsIsIntegrable}).
\end{remark}
%
%
%
%
%
%
%
\subsection{From Non-Integrable Lie Algebras to Non-Integrable Lie Triple Systems}
Given a Banach--Lie algebra $\mathfrak{g}$, we turn it into a Lie triple system $\mathfrak{g}^+$ by equipping it with the triple bracket $[x,y,z]:=\frac{1}{4}[[x,y],z]$. This is motivated by the following fact: A Banach--Lie group $G$ (with Lie algebra $\mathfrak{g}$) can be considered as a symmetric space (denoted by $G^+$) with multiplication $g\cdot h :=gh^{-1}g$ (\cite[Ex.~3.9]{Nee02Cartan}), whose Lie triple system $\Lts(G^+)$ is given by $\mathfrak{g}^+$ (cf.\ arguments of \cite[p.~81]{Loo69}). The exponential maps $\exp_G$ and $\Exp_{G^+}$ coincide (cf.\ \cite[p.~88]{Loo69} or \cite[Ex.~3.9]{Nee02Cartan}).

Endowing $\mathfrak{g}\times \mathfrak{g}$ with the flip involution $\theta\colon\mathfrak{g}\times\mathfrak{g}\rightarrow \mathfrak{g}\times\mathfrak{g}$, $(x,y)\mapsto (y,x)$ leads to the Lie triple system $(\mathfrak{g}\times\mathfrak{g})_-=\{(x,-x)\in\mathfrak{g}\times\mathfrak{g}\}$, which is isomorphic to $\mathfrak{g}^+$ via the isomorphism $\Phi\colon(\mathfrak{g}\times\mathfrak{g})_-\rightarrow \mathfrak{g}^+$, $(x,-x)\mapsto 2x$.
\begin{lemma}
	Given a Banach--Lie algebra $\mathfrak{g}$, we have $2\Pi(\mathfrak{g})\subseteq \Pi(\mathfrak{g}^+)$.
\end{lemma}
\begin{proof}
	From Proposition~\ref{prop:Pi(g)_-inPi(m)}, we know that $\Pi(\mathfrak{g}\times\mathfrak{g})_-\subseteq \Pi((\mathfrak{g}\times\mathfrak{g})_-)$. Via the isomorphism $\Phi$, we deduce that $\Phi(\Pi(\mathfrak{g}\times\mathfrak{g})_-)\subseteq \Pi(\mathfrak{g}^+)$. Since $\Pi(\mathfrak{g}\times\mathfrak{g})_-=\{(x,-x)\in\Pi(\mathfrak{g})\times\Pi(\mathfrak{g})\}$ (cf.\ Remark~\ref{rem:periodSpaceOfDirectProduct}) is mapped by $\Phi$ to $2\Pi(\mathfrak{g})$, we obtain the assertion.
\end{proof}
\begin{proposition}
	A Banach--Lie algebra $\mathfrak{g}$ is integrable if and only if the Lie triple system $\mathfrak{g}^+$ is integrable.
\end{proposition}
\begin{proof}
	If $\mathfrak{g}$ is integrable to a Lie group $G$, then $\mathfrak{g}^+$ is integrable to $G^+$. Conversely, let $\mathfrak{g}^+$ be integrable. Then, by the Integrability Criterion, the period group $\Pi(\mathfrak{g}^+)$ is discrete, so that also the subset $2\Pi(\mathfrak{g})$ is discrete. Thus the period group $\Pi(\mathfrak{g})$ is discrete, so that $\mathfrak{g}$ is integrable.
\end{proof}
\begin{remark}
	Given an integrable Banach--Lie algebra $\mathfrak{g}$, we moreover know that\linebreak $\Pi(\mathfrak{g})\subseteq \Pi(\mathfrak{g}^+)$. Indeed, let $G$ be a 1-connected Lie group with $L(G)=\mathfrak{g}$.
	Observing that
	$$\mathfrak{z}(\mathfrak{g}^+) \ =\ \{x\in \mathfrak{g}^+\colon [x,y]\in\mathfrak{z}(\mathfrak{g}) \text{ for all } y\in\mathfrak{g}\},$$
	we see that $\mathfrak{z}(\mathfrak{g})\subseteq \mathfrak{z}(\mathfrak{g}^+)$. By Remark~\ref{rem:periodGroup=kernel} and Proposition~\ref{prop:periodSpace=kernel}, we thus obtain
	$$\Pi(\mathfrak{g}) \ =\ \ker(\exp_G|_{\mathfrak{z}(\mathfrak{g})}) \ \subseteq\ \ker(\Exp_{G^+}|_{\mathfrak{z}(\mathfrak{g}^+)}) \ =\ \Pi(\mathfrak{g}^+),$$
	since the exponential maps $\exp_G$ and $\Exp_{G^+}$ coincide.
\end{remark}
%
%
%
%
%
%
%
\subsection{Integrability of Complexifications of Real Banach--Lie Algebras}
In this subsection, where also complex Banach--Lie algebras are considered, we shall frequently speak of \emph{real} Banach--Lie algebras (which were simply called Banach--Lie algebras in earlier sections).
\begin{theorem} 
	Let $\mathfrak{g}$ be a real Banach--Lie algebra and $\mathfrak{g}_\CC = \mathfrak{g}\oplus i\mathfrak{g}$ its complexification. Then the following are equivalent:
	\begin{enumerate}
		\item[\rm (a)] The complex Lie algebra $\mathfrak{g}_\CC$ is integrable.
		\item[\rm (b)] The real Lie algebra $\mathfrak{g}$ and the (real) Lie triple system $i\mathfrak{g}$ both are integrable.
	\end{enumerate}
\end{theorem}
\begin{proof}
	Noting that a complex Lie algebra is integrable if and only if its underlying real Lie algebra is integrable (cf.\ \cite[Prop.~5 in III.6.3]{Bou89LieGroups}), we consider $\mathfrak{g}_\CC$ simply as a real Lie algebra, endow it with the complex conjugation and note that $i\mathfrak{g}=(\mathfrak{g}_\CC)_-$.
	
	(a)$\Rightarrow$(b): Being a real subalgebra of the integrable Lie algebra $\mathfrak{g}_\CC$, the Lie algebra $\mathfrak{g}$ is integrable. By Lemma~\ref{lem:integrableSymLieAlgebra}, $i\mathfrak{g}$ is integrable, since $\mathfrak{g}_\CC$ is so.
	
	(b)$\Rightarrow$(a):
	To see that $\mathfrak{g}_\CC$ is integrable, we shall verify the discreteness of the period group $\Pi(\mathfrak{g}_\CC)$ by checking that $\Pi(\mathfrak{g}_\CC)_+$ and $\Pi(\mathfrak{g}_\CC)_-$ are discrete (cf.\ Remark~\ref{rem:im(f)LieSubgroupIffPi(g)Discrete}).
	From $\Pi(\mathfrak{g}_\CC)_-\subseteq\ \Pi((\mathfrak{g}_\CC)_-)$ (cf.\ Proposition~\ref{prop:Pi(g)_-inPi(m)}), we deduce that $\Pi(\mathfrak{g}_\CC)_-$ is discrete, since $\Pi((\mathfrak{g}_\CC)_-)$ is discrete by the Integrability Criterion. To show the discreteness of $\Pi(\mathfrak{g}_\CC)_+$, we consider the quotient morphism $q\colon\mathfrak{g}_\CC\rightarrow \mathfrak{g}\oplus i\mathfrak{g}_{\ad}$ induced by the ideal $i\mathfrak{z}(\mathfrak{g})\unlhd \mathfrak{g}_\CC$. It is easy to see that $\mathfrak{z}(\mathfrak{g}_\CC) = \mathfrak{z}(\mathfrak{g})\oplus i\mathfrak{z}(\mathfrak{g})$ and $\mathfrak{z}(\mathfrak{g}\oplus i\mathfrak{g}_{\ad}) = \mathfrak{z}(\mathfrak{g})$, so that $q$ maps center onto center. Since we further have $\ker(q)= i\mathfrak{z}(\mathfrak{g})\subseteq \mathfrak{z}(\mathfrak{g}_\CC)$, Corollary~\ref{cor:A(Pi(m))=Pi(bar m)}(2) applies and entails that $q(\Pi(\mathfrak{g}_\CC)) = \Pi(\mathfrak{g}\oplus i\mathfrak{g}_{\ad})$. Together with $\Pi(\mathfrak{g}_\CC)_+\oplus \Pi(\mathfrak{g}_\CC)_- \subseteq \Pi(\mathfrak{g}_\CC)$ (cf.\ Remark~\ref{rem:periodGroupWithInvolution}), this leads to
	$$\Pi(\mathfrak{g}_\CC)_+ \ =\ q(\Pi(\mathfrak{g}_\CC)_+\oplus \Pi(\mathfrak{g}_\CC)_-) \ \subseteq\ \Pi(\mathfrak{g}\oplus i\mathfrak{g}_{\ad}).$$
	It suffices to show that $\mathfrak{g}\oplus i\mathfrak{g}_{\ad}$ is integrable, since this implies the discreteness of $\Pi(\mathfrak{g}\oplus i\mathfrak{g}_{\ad})$ and hence of $\Pi(\mathfrak{g}_\CC)_+$. 
	With regard to the injective morphism
	$$\mathfrak{g}\oplus i\mathfrak{g}_{\ad} \rightarrow \mathfrak{g} \times (\mathfrak{g}_{\ad}\oplus i\mathfrak{g}_{\ad}),\ x+iy\mapsto \big(x,(x+\mathfrak{z}(\mathfrak{g}))+iy\big),$$
	it suffices to check that $\mathfrak{g}_{\ad}\oplus i\mathfrak{g}_{\ad}$ is integrable, keeping in mind that $\mathfrak{g}$ is integrable by assumption.
	For this, we again inject $\mathfrak{g}_{\ad}\oplus i\mathfrak{g}_{\ad}$ into $\gl(\mathfrak{g})\oplus i\gl(\mathfrak{g})$ via the map induced by
	$$\mathfrak{g}\oplus i\mathfrak{g}\rightarrow \gl(\mathfrak{g})\oplus i\gl(\mathfrak{g}),\ x+iy\mapsto [x,\cdot] + i[y,\cdot]$$
	and show that $\gl(\mathfrak{g})\oplus i\gl(\mathfrak{g})$ is integrable.
	In fact, we observe that $\gl(\mathfrak{g})\oplus i\gl(\mathfrak{g})$ is the Lie algebra of the general linear group $\GL(\mathfrak{g}_\CC)$ where $\mathfrak{g}_\CC$ is considered as a \emph{complex} Lie algebra.
\end{proof}
This theorem applies to the complexification of real Banach--Lie groups (cf.\ \cite[Th.~IV.7]{GN03}).
\begin{acknowledgements}
	I am grateful to Karl-Hermann Neeb for his helpful communications and proof reading during my research towards this article. This work was supported by the Technical University of Darmstadt and by the Studienstiftung des deutschen Volkes.
\end{acknowledgements}

\bibliography{paperE}
\bibliographystyle{amsalpha}
\end{document}